\documentclass[letterpaper,11pt,oneside,reqno]{amsart}
\usepackage{amsfonts,amsmath, amssymb,amsthm,amscd}
\usepackage[usenames,dvipsnames]{color}
\usepackage{bm}
\usepackage{mathrsfs}
\usepackage{yfonts}
\usepackage[mpexclude,DIV13]{typearea}
\usepackage{verbatim}
\usepackage{hyperref}
\usepackage{graphicx}
\usepackage[latin1]{inputenc}
\usepackage{latexsym}
\usepackage{lscape}
\usepackage{epsfig}
\usepackage{subfigure}
\usepackage{setspace}
\include{bibtex}

\begin{document}

\bibliographystyle{alpha}
\newcommand{\cn}[1]{\overline{#1}}
\newcommand{\e}[0]{\epsilon}
\newcommand{\EE}{\ensuremath{\mathbb{E}}}
\newcommand{\qq}[1]{(q;q)_{#1}}
\newcommand{\dx}[1]{\ensuremath{\nabla_{#1}}}
\newcommand{\pp}[1]{\ensuremath{\mathbf{p}_{#1}}}
\newcommand{\A}[2]{\ensuremath{\mathcal{A}^{(#1)}_{#2}}}
\newcommand{\LL}[2]{\ensuremath{\mathcal{L}^{(#1)}_{#2}}}
\newcommand{\gap}[1]{\ensuremath{\mathrm{gap}_{#1}}}
\newcommand{\Wk}{\ensuremath{\mathbb{W}^k_{\geq 0}}}
\newcommand{\GT}{\ensuremath{\mathbb{GT}}}
\newcommand{\link}{\ensuremath{Q}}
\newcommand{\PP}{\ensuremath{\mathbb{P}}}
\newcommand{\frakP}{\ensuremath{\mathfrak{P}}}
\newcommand{\frakQ}{\ensuremath{\mathfrak{Q}}}
\newcommand{\frakq}{\ensuremath{\mathfrak{q}}}
\newcommand{\R}{\ensuremath{\mathbb{R}}}
\newcommand{\Rplus}{\ensuremath{\mathbb{R}_{+}}}
\newcommand{\C}{\ensuremath{\mathbb{C}}}
\newcommand{\Z}{\ensuremath{\mathbb{Z}}}
\newcommand{\Weyl}[1]{\ensuremath{\mathbb{W}}^{#1}}
\newcommand{\Zgzero}{\ensuremath{\mathbb{Z}_{>0}}}
\newcommand{\Zgeqzero}{\ensuremath{\mathbb{Z}_{\geq 0}}}
\newcommand{\Zleqzero}{\ensuremath{\mathbb{Z}_{\leq 0}}}
\newcommand{\Q}{\ensuremath{\mathbb{Q}}}
\newcommand{\T}{\ensuremath{\mathbb{T}}}
\newcommand{\Y}{\ensuremath{\mathbb{Y}}}
\newcommand{\M}{\ensuremath{\mathbf{M}}}
\newcommand{\MM}{\ensuremath{\mathbf{MM}}}
\newcommand{\W}[1]{\ensuremath{\mathbf{W}}_{(#1)}}
\newcommand{\WM}[1]{\ensuremath{\mathbf{WM}}_{(#1)}}
\newcommand{\Zsd}{\ensuremath{\mathbf{Z}}}
\newcommand{\Fsd}{\ensuremath{\mathbf{F}}}
\newcommand{\symBM}{\ensuremath{\mathbf{W}}}
\newcommand{\symFE}{\ensuremath{\mathbf{S}}}

\newcommand{\Real}{\ensuremath{\mathrm{Re}}}
\newcommand{\Imag}{\ensuremath{\mathrm{Im}}}
\newcommand{\re}{\ensuremath{\mathrm{Re}}}

\newcommand{\Sym}{\ensuremath{\mathrm{Sym}}}

\newcommand{\bfone}{\ensuremath{\mathbf{1}}}

\newcommand{\OO}[0]{\Omega}
\newcommand{\F}[0]{\mathfrak{F}}
\newcommand{\poly}[0]{R}
\def \Ai {{\rm Ai}}
\def \sgn {{\rm sgn}}
\def \SS {\mathcal{S}}
\newcommand{\poles}{\mathbb{A}}
\def \ss {\mathcal{X}}
\newcommand{\var}{{\rm var}}

\newcommand{\ul}[2]{\underline{#1}_{#2}}
\newcommand{\qhat}[1]{\widehat{#1}^{q}}
\newcommand{\La}[0]{\Lambda}
\newcommand{\la}[0]{\lambda}
\newcommand{\ta}[0]{\theta}
\newcommand{\w}[0]{\omega}
\newcommand{\ra}[0]{\rightarrow}
\newcommand{\vectoro}{\overline}
\newtheorem{theorem}{Theorem}[section]
\newtheorem{partialtheorem}{Partial Theorem}[section]
\newtheorem{conj}[theorem]{Conjecture}
\newtheorem{lemma}[theorem]{Lemma}
\newtheorem{proposition}[theorem]{Proposition}
\newtheorem{corollary}[theorem]{Corollary}
\newtheorem{claim}[theorem]{Claim}
\newtheorem{formal}[theorem]{Critical point derivation}
\newtheorem{experiment}[theorem]{Experimental Result}
\newtheorem{prop}{Proposition}

\def\todo#1{\marginpar{\raggedright\footnotesize #1}}
\def\change#1{{\color{green}\todo{change}#1}}
\def\note#1{\textup{\textsf{\color{blue}(#1)}}}

\theoremstyle{definition}
\newtheorem{remark}[theorem]{Remark}

\theoremstyle{definition}
\newtheorem{example}[theorem]{Example}

\theoremstyle{definition}
\newtheorem{definition}[theorem]{Definition}

\theoremstyle{definition}
\newtheorem{definitions}[theorem]{Definitions}

\begin{abstract}
We introduce two new exactly solvable (stochastic) interacting particle systems which are discrete time versions of $q$-TASEP. We call these geometric and Bernoulli discrete time $q$-TASEP. We obtain concise formulas for expectations of a large enough class of observables of the systems to completely characterize their fixed time distributions when started from step initial condition. We then extract Fredholm determinant formulas for the marginal distribution of the location of any given particle.

Underlying this work is the fact that these expectations solve closed systems of difference equations which can be rewritten as free evolution equations with $k-1$ two-body boundary conditions -- discrete $q$-deformed versions of the quantum delta Bose gas. These can be solved via a nested contour integral ansatz. The same solutions also arise in the study of Macdonald processes, and we show how the systems of equations our expectations solve are equivalent to certain commutation relations involving the Macdonald first difference operator.
\end{abstract}

\title{Discrete time $q$-TASEPs}
\author[A. Borodin]{Alexei Borodin}
\address{A. Borodin,
Massachusetts Institute of Technology,
Department of Mathematics,
77 Massachusetts Avenue, Cambridge, MA 02139-4307, USA\newline
Institute for Information Transmission Problems, Bolshoy Karetny per. 19, Moscow 127994, Russia}
\email{borodin@math.mit.edu}

\author[I. Corwin]{Ivan Corwin}
\address{I. Corwin,
Massachusetts Institute of Technology,
Department of Mathematics,
77 Massachusetts Avenue, Cambridge, MA 02139-4307, USA\newline
Clay Mathematics Institute, 10 Memorial Blvd. Suite 902, Providence, RI 02903, USA}
\email{icorwin@mit.edu}

\maketitle

%\setcounter{tocdepth}{1}
%\tableofcontents
%\hypersetup{linktocpage}

\section{Introduction}
The purpose of this paper is to introduce and analyze two exactly solvable discrete time versions of $q$-TASEP. Before introducing them we recall the continuous time Poisson $q$-TASEP which was previously studied in \cite{BorCor,BCS}.

Let us fix some notation used throughout. For $N\geq 1$ we denote the state of $q$-TASEP with $N$ particles as $\vec{x}(t)  = \big(+\infty\equiv x_0(t)>x_1(t)> x_2(t) > \cdots >x_N(t)\big)\in \Z^N$, where we have fixed a virtual particle at infinity by setting $x_0(t)\equiv +\infty$. The gap between particle $i$ and $i-1$ is denoted $\gap{i}(t) := x_{i-1}(t)-x_{i}(t) -1$.

\subsection{Continuous time Poisson $q$-TASEP}
We define and provide some background for the continuous time Poisson $q$-TASEP. All of the results of this paper will be stated simultaneously for this process as well as its two discrete time variants.

\begin{definition}
The $N$-particle continuous time Poisson $q$-TASEP with particle rate parameters $a_1,\ldots ,a_N>0$ is an interacting particle system $\vec{x}(t)$ in which for each $1\leq i\leq N$, $x_i(t)$ moves to $x_i(t)+1$ at exponential rate $a_i(1-q^{\gap{i}(t)})$ (which vanishes when $\gap{i}(t)=0$), where $q$ is a parameter in $(0,1)$. Here we assume that these exponentially distributed jumping events are all independent, and we note that since they occur in continuous time, no two jumps occur simultaneously, almost surely. Also note that the evolution of $x_i(\cdot)$ only depends on those particles with lower indices than $i$. Step initial condition is defined as setting $x_i(0)=-i$ for $1\leq i\leq N$.
\end{definition}

\begin{remark}\label{zerorem}
The continuous time Poisson $q$-TASEP and the recognition of its exact solvability comes from the work of \cite{BorCor} on Macdonald processes. It was soon after studied in its own right in \cite{BCS} via the type of many body system approach methods which we develop herein for the two discrete time variants of the process. Recently \cite{OConPei} showed how $q$-TASEP arises from considerations involving a $q$-deformed generalization of the dual RSK correspondence and \cite{BorPet} showed how it arises from a certain class of nearest-neighbor stochastic dynamics on Gelfand-Tsetlin patterns. It would be very interesting to extend those works or the method of \cite{COSZ} to include the (more general) discrete time $q$-TASEPs of the present paper as well.

The evolution of $\gap{i}(t) = x_{i-1}(t)-x_{i}(t) -1$ for $1\leq i\leq N$ (for the continuous time Poisson $q$-TASEP) is given by a totally asymmetric zero range process with site dependent jump rate $g_i(k)= a_i(1-q^k)$ and infinite sink and source at the boundary. A variant of this zero range process can be seen as corresponding to a particular representation of the $q$-Boson Hamiltonian considered in \cite{SasWad}. In \cite{SasWad}, integrability (in the form of $L$ and $R$ matrices satisfying Yang-Baxter relations) for this $q$-Boson Hamiltonian was shown, though the connection of that to the exact solvability of $q$-TASEP discussed below remains to be understood. A stationary (infinite lattice) variant of the above gap zero range process was also discussed immediately after Theorem 2.9 of \cite{BKS} and an $O(t^{2/3})$ upper and lower bound on the variance of the stationary current is established therein.
\end{remark}

Poisson $q$-TASEP has an interpretation as a model for traffic on a one-lane (discrete) road $\Z$ in which the rate at which cars jump forward is modulated by the distance to the next car (as well as a car dependent rate parameter $a_i$). As the distance goes to zero, the jump rate goes to zero, and as the distance grows, the jump rate approaches $a_i$. Step initial condition corresponds to an initially jammed configuration.

When $q\to 0$, continuous time Poisson $q$-TASEP becomes the well-studied model of continuous time TASEP (and likewise the discrete time versions of $q$-TASEP we introduce become known discrete time version of TASEP). The configuration of particles in TASEP for a fixed time $t$ can be described as a determinantal point process in which all correlation functions are given by minors of a single correlation kernel (cf. \cite{KJ,BorFer} and references therein). In other terminology, TASEP is free-fermion, or related to Schur processes, or non-interacting line ensembles. Given that structure there exists a clean path to Fredholm determinant formulas for marginal distributions, which in turn allow readily for asymptotic analysis (cf. \cite{BorChap,AnderChap} and references therein).

The present work, where $q\in (0,1)$, does not appear to be determinantal, hence new ideas are necessary to study and extract asymptotic distributional information about the processes considered. In recent years there has been a flurry of activity surrounding the analysis of non-determinantal yet still exactly solvable stochastic interacting particle systems \cite{TW1,TW2,TW3,SeppLog,ACQ,SaSp,Dot,CDR,OCon,COSZ,BorCor,BCS,OConPei,BorPet}. In \cite{BorCor,BCS} it was explained how, instead of the determinantal structure of correlation functions, these non-determinantal systems are exactly solvable due to the existence of concise and exact formulas for expectations of large classes of observables of the particle systems (in fact, in many cases large enough classes to uniquely characterize the fixed time distribution of the particle systems).

A natural question about $q$-TASEP is to compute the distribution of $x_n(t)$ (which is the location of particle $n$ at time $t$). There are (presently) two approaches to compute this distribution. The first is through the theory of Macdonald processes \cite{BorCor}, and the second is through the many body system (or duality) approach \cite{BCS}. 

The first approach is entirely based on the integrable properties of Macdonald polynomials (see Section \ref{MacSec} for a brief review on the relationship of the two approaches), and these properties naturally lead to the discovery of $q$-TASEP as well as the computation of nested contour integral formulas for expectations of certain observables of the process. In particular, for step initial condition this shows that (see also \cite{BCGS} or Theorem \ref{momthm} below) for $n_1\geq n_2\geq \cdots \geq n_k>0$,
\begin{equation}\label{exp1}
\EE\left[ \prod_{i=1}^{k} q^{x_{n_i}(t)+n_i} \right]= \frac{(-1)^k q^{k(k-1)/2}}{(2\pi \iota)^k} \int \cdots \int \prod_{1\leq A<B\leq k} \frac{z_A-z_B}{z_A-qz_B} \prod_{j=1}^{k} \left(\prod_{i=1}^{n_j} \frac{a_i}{a_i-z_j}\right) e^{(q-1)tz_j} \frac{dz_j}{z_j},
\end{equation}
where the contour of integration for $z_A$ contains $a_1,\ldots, a_N$ as well as $\{qz_B\}_{B>A}$, but not zero (see Figure \ref{circontours} for an example of such a contour when all $a_i\equiv 1$). The above formula is given in Theorem \ref{momthm} along with parallel formulas for the discrete time versions.

Since the observables $q^{x_n(t)+n}$ are all in $(0,1)$, their moment problem is well-posed. This means that the above formulas uniquely characterize the joint distribution of $q$-TASEP at time $t$ (provided it is started from step initial condition at time zero). It remains a challenge to extract manageable (from the perspective of large $t$ and $n$ asymptotics) formulas for joint distributions, though it is understood how to derive a Fredholm determinant formula which characterizes the one-point distribution of $x_n(t)$. This is described in Theorem \ref{distthm}.

The second approach to solving the continuous time Poisson $q$-TASEP is based solely on the observation that the dynamics of the particle system implies that the expectations on the left-hand side of (\ref{exp1}) satisfy certain closed systems of coupled ODEs, which we call {\it many body systems}. These systems have unique solutions and it is easy to check that the nested contour integral formulas on the right-hand side of (\ref{exp1}) satisfy these systems (and hence are equal to the expectations on the left-hand side of (\ref{exp1})). This second approach is more direct (it avoids explicit mention of Macdonald polynomials) and appears to be more general (it applies to ASEP \cite{BCS}, which is not known to fit into the theory of Macdonald processes). On the other hand, it requires non-trivial input of a suitable interacting particle system, the correct observables to study and some inspiration in guessing a closed form solution to the associated many body system. For all of the versions of $q$-TASEP the fact that these expectations satisfy the many body system can be seen immediately from the Macdonald process perspective as a consequence of certain commutation relations involving Macdonald difference operators (see Section \ref{MacSec}).

It is this second approach which we employ in this paper to study the two exactly solvable discrete time $q$-TASEPs we now introduce.

Before we proceed, let us fix some additional notation. We use $a_1,\ldots, a_N>0$ to denote particle rate parameters, and $\alpha_1,\alpha_2,\ldots\in (0,1)$ and $\beta_1,\beta_2,\ldots\in (0,\infty)$ to denote (discrete) time dependent jump parameters. The indicator function of an event $E$ is written $\mathbf{1}\{E\}$. The $q$-Pochhammer symbol is defined as
$$
(a;q)_{n} = \prod_{i=0}^{n-1} (1-aq^i), \qquad (a;q)_{\infty} = \prod_{i\geq 0} (1-a q^i).
$$
Finally, we define a variant of the Weyl chamber as
$$\Wk = \big\{\vec{n}= (n_1,\ldots,n_k)\in \Z^k: N\geq n_1\geq n_2\geq \cdots \geq n_k\geq 0\big\}.$$

\subsection{Discrete time $q$-TASEPs}

The two discrete time versions of $q$-TASEP which we introduce and study were not fabricated out of thin air, but rather came from further investigations into dynamics related to Macdonald processes which may be detailed in a future work (cf. \cite{BorPet} for some related developments). All three versions of $q$-TASEP are illustrated in Figure \ref{qTASEPs}.

\begin{figure}
\begin{center}
\includegraphics[scale=.8]{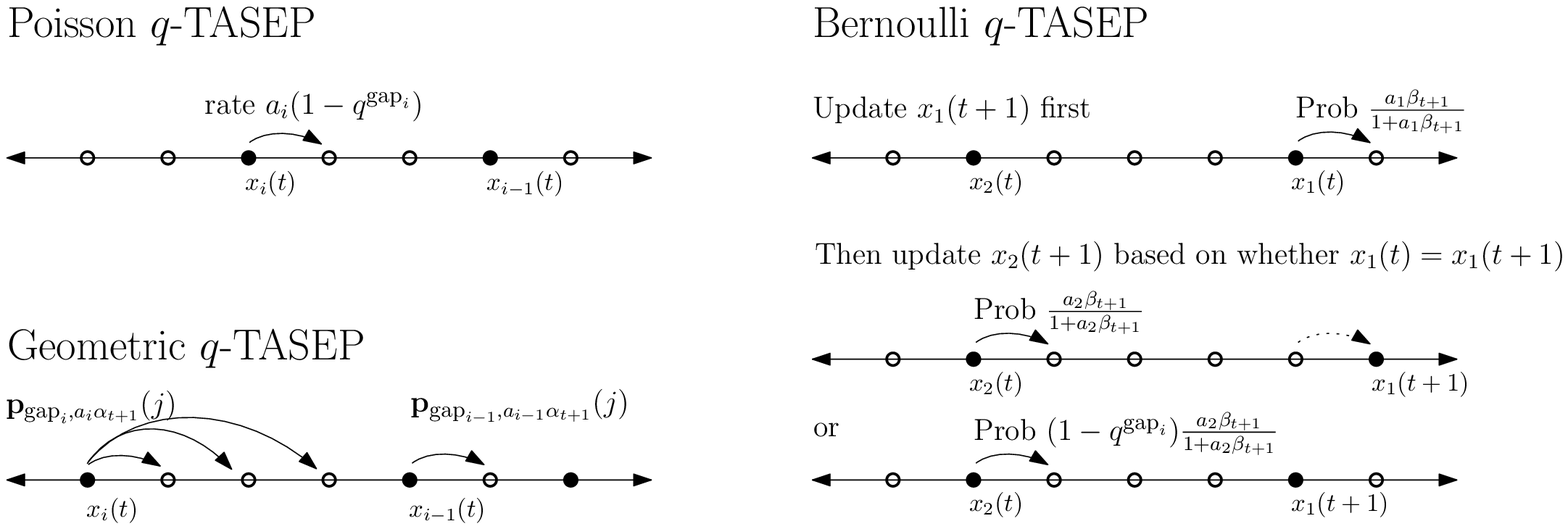}
\end{center}
\caption{Continuous time Poisson $q$-TASEP, discrete time geometric $q$-TASEP with parallel update, and discrete time Bernoulli $q$-TASEP with right-to-left sequential update.}\label{qTASEPs}
\end{figure}

\subsubsection{Discrete time geometric $q$-TASEP}

We will make use of the following $q$-deformation of the truncated geometric distribution. For $m\in \Z_{\geq 0}$ and $\alpha\in (0,1)$ the distribution $\pp{m,\alpha}(j)$  is defined as
\begin{equation*}
\pp{m,\alpha}(j)  = \alpha^j (\alpha;q)_{m-j} \frac{(q;q)_{m}}{(q;q)_{m-j}(q;q)_j} \mathbf{1}\{0\leq j\leq m\}.
\end{equation*}
For $m=+\infty$ we extend the definition so that
\begin{equation*}
\pp{+\infty,\alpha}(j)  = \alpha^j (\alpha;q)_{\infty} \frac{1}{(q;q)_j} \mathbf{1}\{0\leq j\}.
\end{equation*}
\begin{lemma}\label{sumone}
For $m\in \Z_{\geq 0} \cup \{+\infty\}$ and $\alpha\in (0,1)$, $\pp{m,\alpha}(j) \geq 0$ for all $j$ and
\begin{equation*}
\sum_{j=0}^{m} \pp{m,\alpha}(j) = 1.
\end{equation*}
\end{lemma}
\begin{proof}
The positivity is immediate as $\alpha,q\in (0,1)$. That these $\pp{m,\alpha}(j)$ sum to one can be seen inductively (in $m$) by using the recursion (see (10.0.3) in \cite{AAR}) for the $q$-Binomial coefficients (which occur in the definition of $\pp{m,\alpha}(j)$ as the fraction).
\end{proof}

We may now introduce the first discrete time version of $q$-TASEP. We work in the greatest generality for which this is exactly solvable. For simplicity, a reader may want to think of all parameters $a_i\equiv 1$ and all $\alpha_t \equiv \alpha$.

\begin{definition}
The $N$-particle discrete time geometric $q$-TASEP with particle rate parameters $a_1,\ldots, a_N>0$ and time dependent jump parameters $\alpha_1,\alpha_2,\ldots \in (0,1)$ is an interacting particle system $\vec{x}(t)$ in discrete time $t\in \Z_{\geq 0}$. At time $t+1$ the system stochastically evolves from its state at time $t$ according to the following {\it parallel} update rule: For each $1\leq i\leq N$,
\begin{equation*}
\PP\big(x_{i}(t+1) = x_{i}(t) + j \big\vert \gap{i}(t)\big) = \pp{\gap{i}(t),a_i \alpha_{t+1}}(j)
\end{equation*}
where the update is independent for each $i$ and $t$. Since we observe the convention of $x_0(t)=+\infty$, in the above update rule $\gap{1}(t)\equiv +\infty$. We assume throughout the paper at the $a_i$ and $\alpha_t$ are such that $\alpha_i\alpha_t < 1$ for all $i$ and $t$, otherwise the process is not well-defined.
\end{definition}

These dynamics mean that at time $t+1$ each particle $x_i$ hops to a random location in the set $\{x_i(t), x_i(t)+1,\ldots, x_{i-1}(t)-1\}$ with hop length distributed according to independent random variables with distribution $\pp{\gap{i}(t),a_i\alpha_{t+1}}$. These dynamics clearly preserve the order of particles.

\begin{remark}
When $q\to 0$ the jump probabilities become $$\pp{\gap{i}(t),a_i\alpha_{t+1}}(j) = (a_i\alpha_{t+1})^j(1-a_i\alpha_{t+1}) \mathbf{1}\{1\leq j<\gap{i}(t)\} + (a_i\alpha_{t+1})^{\gap{i}(t)} \mathbf{1}\{j=\gap{i}(t)\}.$$ This corresponds with a geometric jump of rate $a_i\alpha_{t+1}$ in which all of the weight given to jumps $j\geq \gap{i}(t)$ are collapsed onto the weight of $j=\gap{i}(t)$. In other words, in parallel each particle attempts a geometric jump to the right subject to blocking by the previous location of the next particle. This process was studied previously in \cite{WarWind} as a marginal of dynamics on Gelfand-Tsetlin patterns which preserve the class of Schur processes.
\end{remark}

\subsubsection{Discrete time Bernoulli $q$-TASEP}

\begin{definition}
The $N$-particle discrete time Bernoulli $q$-TASEP with particle rate parameters $a_1,\ldots, a_N>0$ and time dependent jump parameters $\beta_1,\beta_2,\ldots \in (0,\infty)$ is an interacting particle system $\vec{x}(t)$ in discrete time $t\in \Z_{\geq 0}$. At time $t+1$ the system stochastically evolves from its state at time $t$ according to the following {\it sequential} update rules: First update $x_1(t+1)$ according to
$$
\PP\Big(x_{1}(t+1) = x_{1}(t) + b\Big) = \begin{cases} \frac{1}{1+a_1\beta_{t+1}} &\textrm{for } b=0,\\[.5em] \frac{a_1\beta_{t+1}}{1+a_1\beta_{t+1}} &\textrm{for } b=1,\\[.5em] 0 &\textrm{otherwise}.\end{cases}
$$
Then, for $i=2$ through $N$ (in that order) if $x_{i-1}(t+1)-x_{i-1}(t)=1$ (i.e., the previous jump occurred) then update
$$
\PP\Big(x_{i}(t+1) = x_{i}(t) + b\Big) = \begin{cases} \frac{1}{1+a_i\beta_{t+1}} &\textrm{for } b=0,\\[.5em] \frac{a_i\beta_{t+1}}{1+a_i\beta_{t+1}} &\textrm{for } b=1,\\[.5em] 0 &\textrm{otherwise}.\end{cases}
$$
If $x_{i-1}(t+1)-x_{i-1}(t)=0$ (i.e., the previous jump did not occur) then update
$$
\PP\Big(x_{i}(t+1) = x_{i}(t) + b\Big) = \begin{cases} \frac{1 + a_i\beta_{t+1} q^{\gap{i}(t)}}{1+a_i\beta_{t+1}} &\textrm{for } b=0,\\[.5em] (1-q^{\gap{i}(t)})\frac{a_i\beta_{t+1}}{1+a_i\beta_{t+1}} &\textrm{for } b=1,\\[.5em] 0 &\textrm{otherwise}.\end{cases}
$$
The updates are all independently distributed.
\end{definition}

These dynamics mean that $x_1$ jumps right by one with probability $\tfrac{a_1\beta_{t+1}}{1+a_1\beta_{t+1}}$ and otherwise stays put. Sequentially, if $x_{i-1}$ jumped, then $x_i$ jumps right by one with probability  $\tfrac{a_i\beta_{t+1}}{1+a_i\beta_{t+1}}$, otherwise if $x_{i-1}$ stayed put, then $x_i$ jumps right by one with probability $(1-q^{\gap{i}(t)})\tfrac{a_i\beta_{t+1}}{1+a_i\beta_{t+1}}$. These dynamics clearly preserve the order of particles.

\begin{remark}
When $q\to 0$ the update rule becomes the following: At time $t+1$, particle $x_1$ jumps to the right by 1 with probability $\tfrac{a_1\beta_{t+1}}{1+a_1\beta_{t+1}}$ and stays put otherwise; then sequentially for $x_2$ through $x_N$, the particle $x_i$ jumps to the right by 1 with probability $\tfrac{a_i\beta_{t+1}}{1+a_i\beta_{t+1}}$ if the destination site is unoccupied, otherwise it stays put. To be clear, if $x_1(t)=5$ and $x_2(t)=4$ (neighbors) and $x_1$ jumps right so that $x_1(t+1)=6$, then the rightward jump for $x_2$ is available and $\PP(x_2(t+1)=5)= \tfrac{a_2\beta_{t+1}}{1+a_2\beta_{t+1}}$. This sequential discrete time TASEP was studied previously in \cite{BorFer} as a marginal of dynamics on Gelfand-Tsetlin patterns which preserve the class of Schur processes.
\end{remark}

\subsubsection{Scaling limits}
As $\alpha$ and $\beta$ go to zero and time is suitably rescaled, the discrete time $q$-TASEPs both converge to the continuous time Poisson version. That is to say (taking all $\alpha_t\equiv \alpha$ and $\beta_t\equiv \beta$) that with $\alpha = (1-q) \e$ and $t=\e^{-1} \tau$ geometric $q$-TASEP $\vec{x}(t)$ converges as a process to continuous time Poisson $q$-TASEP with $\tau$ representing time. The same result holds for Bernoulli $q$-TASEP with $\beta = \e$ and $t=\e^{-1}\tau$.

As $q\to 1$, the dynamics of the continuous time Poisson $q$-TASEP converge (under appropriate scaling) to the solution of the semi-discrete stochastic heat equation (equivalently the O'Connell-Yor directed polymer \cite{OY,OCon}) -- see Theorem 4.1.26 of \cite{BorCor} or Proposition 6.2 of \cite{BCS}. This semi-discrete stochastic heat equation converges under diffusive space/time scaling and weak noise scaling to the continuum stochastic heat equation \cite{QRMF,AKQ,ACQ}. It is also shown in \cite{QRMF} that the continuous time Poisson $q$-TASEP has a direct limit to the solution to the continuum stochastic heat equation. Presumably a similar sequence of limits for the discrete time $q$-TASEP dynamics should exist. It would be interesting to investigate whether there exist other degenerations of these discrete time $q$-TASEPs besides those just mentioned. One could, for instance, keep the $\alpha$ or $\beta$ parameters fixed and scaling $q\to 1$.  At least for the geometric case, one expects to make some contact with the work of \cite{COSZ}.

\subsection{Moment formulas and distribution functions}
The three versions of $q$-TASEP share the following three surprising properties, which amount to their exact solvability:

\begin{enumerate}
\item The expectations of observables $\prod_{i=1}^{k} q^{x_{n_i}(t)+n_i}$ for $\vec{n} \in \Wk$ evolve according to closed systems of coupled ODEs (in the continuous time case) or coupled difference equations (in the discrete time case). We call these systems {\it true evolution equations}.
\item The true evolution equations are almost constant coefficient or separable, except for effects which arise when some of the $n_i$ are close together. Rather than trying to solve the true evolution equation directly, one can look for solutions to the free evolution equation (i.e., the constant coefficient and separable system neglecting the boundary effects) on the larger space $\vec{n}\in \Z_{\geq 0}^k$ which have the right initial data when restricted to $\Wk$ and which satisfy certain boundary conditions when some of the $n_i$ are all equal. The restriction of such a solution to $\Wk$ will coincide with the solution to the true evolution equation. Generally, for every possible combination of clusters of $\vec{n}$ (i.e., strings of equal coordinates) there will be additional boundary conditions which must be satisfied. It turns out that for the three versions of $q$-TASEP we consider, it suffices to consider only $k-1$ two-body boundary conditions (which ends up being the same for all three systems) corresponding to when $n_i=n_{i+1}$. That the two-body boundary conditions imply all many body boundary conditions is the hallmark of {\it integrability} in the language of (quantum) many body systems. This type of reduction goes back to the 1931 work of on diagonalizing the Heisenberg spin chain \cite{Bethe}. Our reduction, under the scaling limit to the continuum stochastic heat equation mentioned above, cincides with that for the quantum delta Bose gas (also known as the Lieb-Liniger model \cite{LL}).
\item A general class of solutions to the free evolution equation exists since it is separable and constant coefficient. It is not immediately clear how to combine solutions from this class so as to additionally satisfy the $k-1$ two-body boundary conditions and the initial data. However, for the initial data corresponding to step initial condition, it is also possible to (easily) check that {\it nested contour integral formulas} (such as in (\ref{exp1})) solve the free evolution equation and satisfy the $k-1$ two-body boundary conditions. One may speculate that these formulas should be related to the {\it Bethe ansatz} for diagonalizing this system (cf. \cite{HO,BorCorhalfspace} for such a relationship for a limiting continuum version of these many body problems). From this it would be possible to produce solutions corresponding to general initial conditions. To our knowledge the Bethe ansatz has not been worked out sufficiently in this context (see however related work of \cite{SasWad,Pov,Takeyama}). %In principle one can directly check that the proposed moment formulas do solve the true evolution equations, but this requires some care since the formulas of the type of (\ref{exp1}) are only valid for ordered $n_i$ and the true evolution equation is not constant coefficient. In fact, if the true evolution equation is not integrable, there is no reason to expect such explicit formulas.
\end{enumerate}

\subsection{Outline}
In Section \ref{results} we record the consequences of the above mentioned properties. The rest of the paper is devoted to proving these results along the lines of the three steps given above. In particular, in Section \ref{truesec} we prove that the expectations of the above mentioned observables do satisfy explicit true evolution equations. In Section \ref{freesec} we demonstrate the integrability of these systems by reducing consideration to free evolution equations with $k-1$ boundary conditions. In Section \ref{checkingformulas} we check that for initial data corresponding to step initial conditions this system is indeed solvable via a nested contour integral formula. In Section \ref{endsec} we also briefly remark on the relationship between the true evolution equations and Markov process duality as well as certain commutation relations involving Macdonald difference operators.

\subsection{Acknowledgements}
AB was partially supported by the NSF grant DMS-1056390. IC was partially supported by the NSF through DMS-1208998 as well as by Microsoft Research through the Schramm Memorial Fellowship, and by a Clay Mathematics Institute Research Fellowship.

\section{Main results}\label{results}

We state in parallel our main results for all three variants of $q$-TASEP. To facilitate this we use $\vec{x}(t)$ for each process and instead use different expectation symbols to denote running the different dynamics.

\begin{theorem}\label{momthm}
Fix $q\in (0,1)$. Consider:
\begin{enumerate}
\item Continuous time $q$-TASEP with particle rate parameters $a_1,\ldots, a_N>0$. Let $\EE^{{\rm Poi}}$ represent the expectation operator for this process started from step initial condition. Define
\begin{equation*}
f^{{\rm Poi}}_{t}(z) = e^{tz};
\end{equation*}
\item Discrete time geometric $q$-TASEP with particle rate parameters $a_1,\ldots, a_N>0$ and time dependent jump parameters $\alpha_1,\alpha_2,\ldots \in (0,1)$. Let $\EE^{{\rm geo}}$ represent the expectation operator for this process started from step initial condition. Define
\begin{equation*}
f^{{\rm geo}}_{t}(z) = \prod_{s=1}^{t}\frac{1}{(\alpha_s z;q)_{\infty}};
\end{equation*}
\item Discrete time Bernoulli $q$-TASEP with particle rate parameters $a_1,\ldots, a_N>0$ and time dependent jump parameters $\beta_1,\beta_2,\ldots \in (0,\infty)$. Let $\EE^{{\rm Ber}}$ represent the expectation operator for this process started from step initial condition. Define
\begin{equation*}
f^{{\rm Ber}}_{t}(z) = \prod_{s=1}^{t}(1+\beta_s z).
\end{equation*}
\end{enumerate}
Fix $k\geq 1$, then for all $\vec{n}\in \Wk$ (i.e., $n_1\geq n_2\geq \cdots \geq n_k\geq 0$) and $\ell\in \big\{{\rm Poi}, {\rm geo}, {\rm Ber}\big\}$,
\begin{equation}\label{momthmeq}
\EE^{\ell} \left[ \prod_{i=1}^{k} q^{x_{n_i}(t)+n_i} \right] = \frac{(-1)^k q^{\frac{k(k-1)}{2}}}{(2\pi \iota)^k} \int \cdots \int \prod_{1\leq A<B\leq k} \frac{z_A-z_B}{z_A-qz_B} \prod_{j=1}^{k} \left(\prod_{i=1}^{n_j}\frac{a_i}{a_i-z_j}\right) \frac{f^{\ell}(qz_j)}{f^{\ell}(z_j)} \frac{dz_j}{z_j},
\end{equation}
where the contour of integration for $z_A$ contains $a_1,\ldots, a_N$, and $\{qz_B\}_{B>A}$ but not 0 or any poles of $\tfrac{f^{\ell}(qz)}{f^{\ell}(z)}$.
\end{theorem}

\begin{figure}
\begin{center}
\includegraphics[scale=1.2]{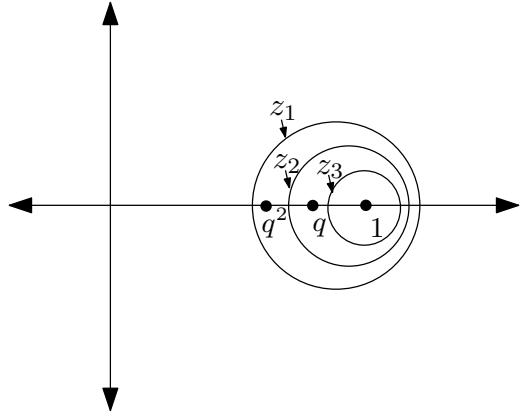}
\end{center}
\caption{Nested contours for $k=3$ and $a_i\equiv 1$.}\label{circontours}
\end{figure}

The remaining sections of this article are concerned with proving the above result (see in particular Section \ref{checkingformulas}).

\begin{remark}\label{arborder}
Consider running an arbitrary combination of the three types of $q$-TASEP with general parameters: Run continuous time Poisson $q$-TASEP for time $\gamma\geq 0$ and discrete time geometric $q$-TASEP with $\alpha_1,\alpha_2,\ldots$ and discrete time Bernoulli $q$-TASEP with $\beta_1,\beta_2,\ldots$. The methods we develop herein imply the following results. The terminal state of the particle system after running combinations of the three dynamics does not depend on the order in which they were run (in fact, one can take turns between running the three different dynamics). Moreover, letting $\vec{x}(t)$ represent the state of the process after having run these three dynamics we find that for all $k\geq 1$ and $\vec{n}\in \Wk$,
\begin{eqnarray*}
\EE \left[ \prod_{i=1}^{k} q^{x_{n_i}(t)+n_i} \right] &=& \frac{(-1)^k q^{\frac{k(k-1)}{2}}}{(2\pi \iota)^k} \int \cdots \int \prod_{1\leq A<B\leq k} \frac{z_A-z_B}{z_A-qz_B} \\
&&\times\,\prod_{j=1}^{k} \left(\prod_{i=1}^{n_j} \frac{a_i}{a_i-z_j}\right) e^{(q-1)\gamma z_j} \prod_{m\geq 1} (1-\alpha_m z_j)\frac{1+q\beta_m z_j}{1+\beta_m z_j} \frac{dz_j}{z_j},
\end{eqnarray*}
where the contour of integration for $z_A$ contains $a_1,\ldots, a_N$, and $\{qz_B\}_{B>A}$ but not 0 or any other poles of the integrand. The above expression is just a product of $\tfrac{f^{\ell}(qz)}{f^{\ell}(z)}$ over $\ell\in \big\{{\rm Poi}, {\rm geo}, {\rm Ber}\big\}$.

The above nested contour integral formulas arise naturally in the study of Macdonald processes (cf. Proposition 3.1.5 of \cite{BorCor} and Corollary 4.7 of \cite{BCGS}) and the existence of the associated integrable quantum many body systems is actually a consequence of this connection (see Section \ref{MacSec} for more details).
\end{remark}

\begin{remark}\label{distrem}
The expectations considered above in Theorem \ref{momthm} uniquely characterize the fixed time joint distribution of $q^{x_1(t)+1},\ldots, q^{x_N(t)+N}$, and hence also that of $\vec{x}(t)$. This is because the $q^{x_i(t)+i}$ are deterministically in $(0,1)$, and thus their moments uniquely identify their distributions \cite{Peter,Kleiber}. In principle it should be possible to extract formulas for any fixed time expectations of $\vec{x}(t)$ from the explicit formulas in Theorem \ref{momthm}. It would be desirable to find moment generating functions (which characterize the joint distributions of $\vec{x}(t)$) for which there are sufficiently compact formulas to allow for asymptotics. Presently it is only known how to do this for the one-point distribution of $x_n(t)$.
\end{remark}

As an application of the above moment formulas it is possible to prove the following Fredholm determinant formula for the $q$-Laplace transform of $q^{x_n(t)+n}$ which in turn characterize the distribution of $x_n(t)$ via a simple inversion formula (see Remark \ref{qinversion}). The procedure for going from the nested contour integral formulas above to the Fredholm determinant formula is developed in \cite{BorCor} and the below result follows immediately from this procedure. For simplicity of contours, we assume here that all $a_i\equiv 1$ (see Remark \ref{notasrem}).

\begin{theorem}\label{distthm}
Fix $q\in (0,1)$ and particle rate parameters $a_i\equiv 1$. Consider the three versions of $q$-TASEP corresponding with $\ell\in \big\{{\rm Poi}, {\rm geo}, {\rm Ber}\big\}$ (in the notation of Theorem \ref{momthm}). Then for all $\zeta\in\C\setminus \R_+$
\begin{equation}\label{MellinBarnes}
\EE^{\ell} \left[\frac{1}{\big(\zeta q^{x_{n(t)+n}};q\big)_{\infty}}\right] = \det(I + K^{\ell}_{\zeta})
\end{equation}
where $\det(I + K^{\ell}_{\zeta})$ is the Fredholm determinant of $K^{\ell}_\zeta: L^2(C_1)\to L^2(C_1)$ for $C_1$ a small positively oriented circle containing 1. The operator $K^{\ell}_\zeta$ is defined in terms of its integral kernel
\begin{equation*}
K^{\ell}_{\zeta}(w,w') = \frac{1}{2\pi \iota} \int_{-\iota \infty + 1/2}^{\iota\infty +1/2} \frac{\pi}{\sin(-\pi s)} (-\zeta)^s \frac{h^{\ell}(q^s w)}{h^{\ell}(w)} \frac{1}{q^s w - w'} ds
\end{equation*}
with
\begin{equation*}
h^{\ell}(w) = (w;q)_{\infty}^n f^{\ell}_t(w)
\end{equation*}
where the function $f^{\ell}_t(w)$ is defined as in Theorem \ref{momthm}.

The following second formula also holds:
\begin{equation}\label{Cauchy}
\EE^{\ell} \left[\frac{1}{\big(\zeta q^{x_{n(t)+n}};q\big)_{\infty}}\right] = \frac{\det(I + \zeta \tilde{K}^{\ell})}{(\zeta;q)_{\infty}}
\end{equation}
where $\det(I + \zeta \tilde{K}^{\ell})$ is the Fredholm determinant of $\zeta$ times the operator $\tilde{K}^{\ell}_\zeta: L^2(C_{0,1})\to L^2(C_{0,1})$ for $C_{0,1}$ a positively oriented circle containing 0 and 1 (and no poles of $f^{\ell}_t(w)$). The operator $\tilde{K}^{\ell}$ is defined in terms of its integral kernel
\begin{equation*}
\tilde{K}^{\ell}(w,w') = \left(\frac{1}{1-w}\right)^n \frac{f^{\ell}_t(qw)}{f^{\ell}_t(w)} \frac{1}{qw'-w}
\end{equation*}
where the function $f^{\ell}_t(w)$ is defined as in Theorem \ref{momthm}.
\end{theorem}

\begin{proof}
We refer to the first Fredholm determinant (\ref{MellinBarnes}) as Mellin-Barnes type, and the second Fredholm determinant (\ref{Cauchy}) as Cauchy type. These Mellin-Barnes type and Cauchy type Fredholm determinant formulas were previously proved in the form of Theorems 3.2.10 and 3.2.16 (respectively) of \cite{BorCor}. However, those results are phrased in terms of Macdonald processes (in which $\langle q^{k \lambda_n}\rangle$ corresponds with $\EE[q^{k(x_n(t)+n)}]$). The general scheme which leads to these results is recorded in Sections 3.1 and 3.2 of \cite{BCS} and shows how to go from the nested contour integral formulas for $\EE^{\ell}\left[q^{k(x_n(t)+n)}\right]$ to the (respective) Mellin-Barnes and Cauchy type Fredholm determinant formulas in the theorem. In the application of Propositions 3.6 and 3.10 of \cite{BCS} it is necessary to check certain technical conditions on the function $f$ to make sure that the formal manipulations which lead to the Fredholm determinants are, in fact, numerical identities. These conditions are checked in the proofs of Theorems 3.2.10 and 3.2.16 (respectively) of \cite{BorCor} so we do not repeat them.
\end{proof}

\begin{remark}\label{notasrem}
The only modification of Theorem \ref{distthm} for general $a_i$ is that in the first Fredholm determinant,
$$h^{\ell}(w) = \prod_{i=1}^{n} (w/a_i;q)_{\infty}\, f^{\ell}_t(w),$$
and consequently some care must be taken in specifying the appropriate choice of contours for the $w$ and $s$ integration (so that the $w$ contour contains the $a_1,\ldots, a_n$ while avoiding other poles). See for example, Theorem 3.2.11 of \cite{BorCor}, or Theorem 4.13 of \cite{BorCorFer}. Likewise, in the second Fredholm determinant formula, the term $\left(\tfrac{1}{1-w}\right)^n$ in defining $\tilde K^{\ell}(w,w')$ is replaced by $\prod_{i=1}^{n} \tfrac{a_i}{a_i-w}$ and the contour $C_{0,1}$ is replaced by a suitable one which contains 0 and all $a_1,\ldots, a_n$.
\end{remark}

\begin{remark}\label{qinversion}
The transform from the probability distribution of $q^{x_n(t)+n}$ to the expectation on the left-hand sides of (\ref{MellinBarnes}) and (\ref{Cauchy}) is called a $q$-Laplace transform. Just as the usual Laplace transform can be inverted via a single contour integral in the spectral variable, so too can we recover the distribution of $q^{x_{n}(t)+n}$ or $x_n(t)+n$ from its transform. To state this inversion, write
\begin{equation*}
G^{\ell}(\zeta) = \EE^{\ell} \left[\frac{1}{\big(\zeta q^{x_{n(t)+n}};q\big)_{\infty}}\right].
\end{equation*}
Then it follows from Proposition 3.1.1 of \cite{BorCor} (see also \cite{Banger}) that
\begin{equation*}
\PP^{\ell}\big(x_{n}(t)+n = m\big) = -q^m  \frac{1}{2\pi \iota} \int_{C_m} ( q^{m+1} \zeta;  q)_{\infty}  G^{\ell}(\zeta) d\zeta
\end{equation*}
where $C_m$ is a positively oriented circle which encircles $\{q^{-j}\}_{0\leq j\leq m+1}$.
\end{remark}

\begin{remark}
There are a number of interesting limit theorem results which should be accessible via asymptotic analysis of the Fredholm determinants in Theorem \ref{distthm}. The Mellin-Barnes type Fredholm determinant seems to be most easily analyzed asymptotically. Presently, the only asymptotics related to $q$-TASEP which have been worked out correspond with first taking the limit for $n$ fixed and $q\to 1$ in which $q$-TASEP converges \cite{BCS} to the semi-discrete stochastic heat equation \cite{OY,OCon}. At that level \cite{BorCor,BorCorFer,BCFV} have proved GUE Tracy-Widom limit theorems (see also \cite{BCR} for a related discrete stochastic heat equation \cite{SeppLog}). There remain a number of limit theorems which should be provable from the Fredholm determinant formulas above, such as the GUE Tracy-Widom limit theorems for the three variants of $q$-TASEP. It is possible that in proving such results slightly different choices of contours (as in \cite{BorCorFer,BCFV}) will be necessary.
\end{remark}

\section{True evolution equations}\label{truesec}
The first surprising property of q-TASEP is that the expectations of the observables $\prod_{i=1}^{k} q^{x_{n_i}(t)+n_i}$ evolve according to closed systems of coupled ODEs (in the continuous time case) or coupled difference equations (in the discrete time case). We call these systems the {\it true evolution equations}. This property is not a generic property of interacting particle systems, but rather something quite special to these systems and these choices of observables.

In order to state the true evolution equations it will be useful to have an alternative notation for $\vec{n}\in \Wk$ wherein $y_i$ counts the number of $n_j=i$ (i.e., the size of the cluster of $n_j$'s equal to $i$). More precisely, define
\begin{eqnarray}\label{yn}
Y^N &=& \big\{\vec{y} = (y_0,y_1,\ldots, y_N)\in \Z_{\geq 0}^{N+1}\big\}\\
Y^N_k &=& \big\{\vec{y}\in Y^N : \sum_{i=0}^{N} y_i = k\big\}.
\end{eqnarray}
To each $\vec{n}\in \Wk$ associate $\vec{y}(\vec{n})\in Y^N_k$ via $y_i(\vec{n}) = |\{j: n_j=i\}|$, and to each $\vec{y}\in Y^N_k$ associate $\vec{n}(\vec{y})\in \Wk$ as the unique $\vec{n}$ such that $\vec{y}(\vec{n}) = \vec{y}$. Thus $\vec{y}(\vec{n})$ lists the multiplicities of each number in $\{0,1,\ldots, N\}$ in $\vec{n}$ and $\vec{n}(\vec{y})$ associates to such multiplicities an ordered list in $\Wk$. For instance if $\vec{n}=(4,4,2,1)$ then $\vec{y}(\vec{n}) = (0,1,1,0,2)$. We say that $\vec{n}$ has $\ell$ clusters if $\ell = |\{i:y_{i}(\vec{n}) >0\}|$. For the present example $\vec{n}=(4,4,2,1)$ has three clusters.

For a vector $\vec{y}\in Y_k^N$ and a vector $\vec{s} = (s_1,\ldots, s_N)$ with $0\leq s_i\leq y_i$ define
\begin{eqnarray*}
\vec{y}^{\,s_i}_{i,i-1} &=& (y_0,\ldots, y_{i-1}+s_i,y_{i}-s_i,\ldots, y_N)\\
\vec{y}^{\,\vec{s}} &=& (y_0+s_1,y_1+s_2-s_1,y_2+s_3-s_2,\ldots, y_{N-1}+s_N-s_{N-1},y_{N}-s_N)
\end{eqnarray*}
Both of these map $\vec{y}$ into another element of $Y^N_k$.

The following combinatorial coefficients are used in defining the true evolution equation.
\begin{definition}
For $a\in \R$, $q\in (0,1)$, $y\in \Z_{\geq 0}$ and integers $1\leq s\leq y$ define
\begin{equation}\label{Cas}
C_{a}(y,s) = (-a)^s (-a;q)_{y-s} \frac{(q;q)_{y}}{(q;q)_{y-s}(q;q)_{s}},
\end{equation}
with the convention that $C_{a}(y,s)=0$ for $s<0$ or $s>y$.
\end{definition}

We record the following easily checkable properties of the coefficients $C_{a}(y,s)$ which will prove useful:
\begin{eqnarray}\label{Cid}
\nonumber(q^s-1)(1+ a q^{y-s}) C_{a}(y,s) &=& a(1-q^{y+1-s})C_{a}(y,s-1) \\
C_{a}(y+1,s) &=& (1+a q^{y-s}) C_{a}(y,s) - a q^{y+1-s} C_{a}(y,s-1)\\
\nonumber C_{qa}(y,s) &=& \frac{1}{1+a}\left( (1+a q^{y-s}) C_{a}(y,s) + a (1-q^{y+1-s}) C_{a}(y,s-1)\right) .
\end{eqnarray}

\begin{definition}
For $a>0$, $1\leq i\leq N$ and $\vec{y}\in Y^N$ define the following difference operators
\begin{equation*}
(\LL{a}{i} f)(\vec{y}) = a(1-q^{y_i}) \big(f(\vec{y}_{i,i-1}^{\,1})-f(\vec{y})\big), \qquad (\A{a}{i} f)(\vec{y}) = \sum_{s_i=0}^{y_i} C_{a}(y_i,s_i) f(\vec{y}_{i,i-1}^{\,s_i}).
\end{equation*}
\end{definition}

\begin{remark}\label{zeroremark}
If $f(\vec{y})=0$ for $\vec{y}$ with $y_0>0$, then it follows that
\begin{equation*}
(\LL{a}{1} f)(\vec{y}) = a(1-q^{y_1}) \big(-f(\vec{y})\big), \qquad (\A{a}{1} f)(\vec{y}) = C_{a}(y_1,0) f(\vec{y}).
\end{equation*}
Moreover, for such an $f$, we have that
\begin{equation}\label{rhsexp}
\A{a_1}{1}\cdots \A{a_N}{N} f(\vec{y}) = C_{a_1}(y_1,0) \sum_{s_2=0}^{y_2} C_{a_2}(y_2,s_2) \cdots \sum_{s_N=0}^{y_N} C_{a_N}(y_N,s_N) f(\vec{y}^{\,\vec{s}})
\end{equation}
with $s_1=0$.
\end{remark}

\begin{definition}\label{truedef}
We say that $h(t;\vec{y})$ solves the (rate parameter $a_1,\ldots, a_N$) continuous time Poisson $q$-TASEP / discrete time Geometric $q$-TASEP (with time dependent jump parameters $\alpha_1,\alpha_2,\ldots~\in~(0,1)$) / discrete time Bernoulli $q$-TASEP (with time dependent jump parameters $\beta_1,\beta_2,\ldots~\in~(0,\infty)$) {\it true evolution equation} with initial data $h_0(\vec{y})$ if
\begin{enumerate}
\item For all $\vec{y}\in Y^N$ and $t\geq 0$,
\begin{eqnarray*}
\frac{d}{dt} h(t;\vec{y}) &=& \sum_{i=1}^{N} \LL{a_i}{i} h(t;\vec{y}) \qquad\qquad\qquad\qquad\qquad\textrm{(Poisson)},\\
h(t+1;\vec{y}) &=& \A{-a_1 \alpha_{t+1}}{1} \cdots \A{-a_N \alpha_{t+1}}{N} h(t;\vec{y}) \qquad\quad\textrm{(geometric)},\\
\A{a_1 \beta_{t+1}}{1} \cdots \A{a_N \beta_{t+1}}{N} h(t+1;\vec{y}) &=& \A{q a_1 \beta_{t+1}}{1} \cdots \A{q a_N \beta_{t+1}}{N} h(t;\vec{y})\qquad\quad\,\textrm{ (Bernoulli)};
\end{eqnarray*}
\item For all $\vec{y}\in Y^N$ such that $y_0>0$, $h(t;\vec{y})\equiv 0$ for all $t\geq 0$;
\item For all $\vec{y}\in Y^N$, $h(0;\vec{y}) = h_0(\vec{y})$.
\end{enumerate}
\end{definition}

\begin{lemma}\label{existuniq}
The above respective true evolution equations have unique solutions.
\end{lemma}
\begin{proof}
The operators $\LL{a}{i}$ and $\A{a}{i}$ map the space of functions $f:Y^N\to \R$ onto itself. In fact, these operators restrict to mapping the space of functions $f:Y^N_k\to \R$ onto itself. On account of this, the above true evolution equations restrict to a collection of closed systems of linear ODEs (for the continuous time case) or difference equations (for the two discrete time cases), indexed by $k$. It suffices, therefore, to prove existence and uniqueness when restricted to the system corresponding to each $k\geq 1$.

For fixed $k$, our system of linear differential or difference equations is also triangular in the following sense. Define a partial ordering on $Y^N_k$ so that $y'\leq y$ if for all $0\leq i\leq N$, $y'_i+\cdots +y'_N\leq y_i+\cdots +y_N$. Then triangularity means for the continuous time case that $\frac{d}{dt} h(t;\vec{y})$ depends only upon $h(t,\vec{y}')$ for $\vec{y}' \leq \vec{y}$, and for the discrete time cases that $h(t+1;\vec{y})$ depends only upon $h(t,\vec{y}')$ for $\vec{y}' \leq \vec{y}$. This is easily seen from the definition of the true evolution equations.

Finally, for each $k\geq 1$, the associated closed, triangular system of linear ODEs / difference equations is also finite. One account of this, one can apply standard methods (such as in \cite{Cod} for the continuous time case or linear algebra for the discrete time cases) to conclude the existence and uniqueness of solutions.
\end{proof}

We may now state the main result of this section, which is that expectations of certain observables of the three variants of $q$-TASEP solve the above (respective) true evolution equations.

\begin{theorem}\label{truethm}
Consider the (rate parameter $a_1,\ldots, a_N$) continuous time Poisson $q$-TASEP / discrete time Geometric $q$-TASEP (with time dependent jump parameters $\alpha_1,\alpha_2,\ldots\in (0,1)$) / discrete time Bernoulli $q$-TASEP (with time dependent jump parameters $\beta_1,\beta_2,\ldots\in (0,\infty)$) started from an arbitrary initial condition $\vec{x}(0) = \vec{x}$. Then for any $k\geq 1$ and $\vec{n}\in \Wk$,
\begin{equation}\label{truelhs}
\EE\left[\prod_{i=0}^k q^{x_{n_i}(t)+n_i}\right]  = h(t;\vec{y}(\vec{n}))
\end{equation}
where $h(t;\vec{y})$ solves the (respective) true evolution equation with initial data
\begin{equation*}
h_0(\vec{y}(\vec{n})) = \EE\left[\prod_{i=0} q^{x_{n_i}+n_i}\right].
\end{equation*}
\end{theorem}

\begin{proof}
Due to Lemma \ref{existuniq} it suffices to show that the left-hand side of (\ref{truelhs}) satisfies the true evolution equation (in each of the three cases).

In what follows let $\vec{y}=\vec{y}(\vec{n})$. It is then convenient to rewrite the left-hand side of (\ref{truelhs}) as
\begin{equation*}
\EE\left[\prod_{i=0}^k q^{x_{n_i}(t)+n_i}\right]= \EE\left[\prod_{i=0}^{N} q^{(x_i(t)+i)y_i}\right]
\end{equation*}
with the convention that when $y_0>0$, the above is zero, and when $y_0=0$, the above product starts at $i=1$. From this one readily sees that condition (2) of the true evolution equation is satisfied. Condition (3) is immediate as well.

We prove condition (1) for each of the three cases. In fact, the Poisson case was previously proved in \cite{BCS}, though we include it here as well for completeness.\newline

\noindent {\bf Poisson case}: Let $L$ denote the generator of continuous time Poisson $q$-TASEP. Then it follows that
\begin{eqnarray*}
\frac{d}{dt} \EE\left[\prod_{i=0}^{N} q^{(x_i(t)+i)y_i}\right] &=& L \EE\left[ \prod_{i=0}^{N} q^{(x_i(t)+i)y_i}\right]= \EE\left[L \prod_{i=0}^{N} q^{(x_i(t)+i)y_i}\right]\\
&=& \EE\left[\sum_{j=1}^{N} a_j(1-q^{x_{j-1}(t)-x_{j}(t)-1})\big(q^{(x_j(t)+1+j)y_j}-q^{(x_j(t)+j)y_j}\big) \prod_{\substack{i=0\\i\neq j}}^{N} q^{(x_i(t)+i)y_i}\right]\\
&=& \EE\left[\sum_{j=1}^{N} a_j(1-q^{x_{j-1}(t)-x_{j}(t)-1})(q^{y_j}-1) \prod_{i=0}^{N} q^{(x_i(t)+i)y_i}\right]\\
&=& \EE\left[\sum_{j=1}^{N} \LL{a_j}{j} \prod_{i=0}^{N} q^{(x_i(t)+i)y_i}\right] = \sum_{j=1}^{N} \LL{a_j}{j} \EE\left[\prod_{i=0}^{N} q^{(x_i(t)+i)y_i}\right].
\end{eqnarray*}
This shows that the expectation in question does satisfy the true evolution equation.

\noindent {\bf Geometric case}: We show the following stronger statement. Let $\sigma^N_t$ be the sigma-field generated by the random variables $x_1(t),\ldots, x_N(t)$. Then, as $\sigma^N_t$ measurable random variables
\begin{equation*}
\EE\left[\prod_{i=0}^N q^{(x_{i}(t+1)+i)y_j} \big| \sigma^N_t\right]  =  \A{-a_1 \alpha_{t+1}}{1} \cdots \A{-a_N \alpha_{t+1}}{N} \prod_{i=0}^{N} q^{(x_i(t)+i)y_i}.
\end{equation*}
Taking expectations of both sides (with respect to the random variables $\vec{x}(t)$) recovers the desired result. In what follows we write $\alpha$ instead of $\alpha_{t+1}$.

In order to show the above result, observe that it follows immediately from the geometric dynamics that
\begin{equation}\label{geolhs}
\EE\left[\prod_{i=0}^N q^{(x_{i}(t+1)+i)y_j} \big| \sigma^N_t\right]  =\prod_{i=1}^{N} \left(\sum_{j=0}^{\gap{i}(t)} \pp{\gap{i}(t), a_i \alpha}(j)  q^{jy_i}\right) \prod_{i=0}^{N} q^{(x_i(t)+i)y_i}.
\end{equation}

\begin{lemma}\label{cqm}
For all $m,y \geq 0$ and $a>0$
\begin{equation}\label{lemmalhs}
\sum_{j=0}^{m} \pp{m, a}(j)  q^{jy} = \sum_{s=0}^{y} C_{-a}(y,s) q^{sm},
\end{equation}
and likewise
\begin{equation*}
\sum_{j=0}^{\infty} \pp{\infty, a}(j)  q^{jy} = C_{-a}(y,0).
\end{equation*}
\end{lemma}

Applying this lemma we find
\begin{eqnarray*}
\textrm{LHS (\ref{geolhs})} &=& C_{-a_1\alpha}(y_1,0) \sum_{s_2=0}^{y_2} C_{-a_2\alpha}(y_2,s_2) q^{\gap{2}(t) s_2} \cdots \sum_{s_N=0}^{y_N} C_{-a_N\alpha}(y_N,s_N) q^{\gap{N}(t) s_N}
\prod_{i=0}^{N} q^{(x_i(t)+i)y_i}\\
&=& C_{-a_1\alpha}(y_1,0) \sum_{s_2=0}^{y_2} C_{-a_2\alpha}(y_2,s_2) \cdots \sum_{s_N=0}^{y_N} C_{-a_N\alpha}(y_N,s_N)
\prod_{i=0}^{N} q^{(x_i(t)+i)(y_i+s_{i+1}-s_i)}
\end{eqnarray*}
with the convention $s_0=s_1=s_{N+1}=0$. The second line follows from the definition of $\gap{i}(t)$. Thus, we have reached a formula for the left-hand side of (\ref{geolhs}) which, in light of equation (\ref{rhsexp}) matches the right-hand side of (\ref{geolhs}).

\begin{proof}[Proof of Lemma \ref{cqm}]
We will prove equation (\ref{lemmalhs}). The $m=+\infty$ case is in the same spirit so we leave it out. Taking $m$ fixed, denote the left-hand side of (\ref{lemmalhs}) as $\mu_y$.

We may apply Lemma \ref{sumone} with the parameter $\alpha$ from its statement taken to be  $q^y a$. This yields
\begin{equation*}
\sum_{j=0}^{m} \pp{m,q^y a}(j) = 1.
\end{equation*}
From the definition of $\pp{m,q^y a}(j)$ this equality can be rewritten as
\begin{equation}\label{secondset}
\frac{1}{(a;q)_y} \sum_{j=0}^{m} q^{yj} \pp{m,a}(j) (q^{m-j}a;q)_y = 1.
\end{equation}
We may expand the product $(q^{m-j}a;q)_y$ as
\begin{equation*}
(q^{m-j}a;q)_y = \sum_{r=0}^{y} (-q^{m-j}a)^r e_r(1,q,\ldots, q^{y-1}).
\end{equation*}
where $e_r$ is the degree $r$ elementary symmetric polynomial (see e.g. Section I.2 of \cite{M}).
Plugging this into (\ref{secondset}) and rearranging the summations of $j$ and $r$ yields
\begin{equation*}
\frac{1}{(a;q)_y} \sum_{r=0}^{y} (-a)^r q^{mr} e_r(1,q,\ldots, q^{y-1}) \sum_{j=0}^{m} q^{(y-r)j} \pp{m,a}(j) = 1.
\end{equation*}
Rewriting in terms of the $\mu_y$, we find
\begin{equation}\label{aboverelation}
\frac{1}{(a;q)_y} \sum_{r=0}^{y} (-a)^r q^{mr} e_r(1,q,\ldots, q^{y-1}) \mu_{y-r} = 1.
\end{equation}
Note also that (see e.g. Section I.3, exercise 1 of \cite{M})
\begin{equation*}
e_r(1,q,\ldots, q^{y-1}) = q^{\frac{r(r-1)}{2}} \frac{(q;q)_y}{(q;q)_r (q;q)_{y-r}}.
\end{equation*}

The relation (\ref{aboverelation}) on the $\mu_y$ uniquely characterizes them, hence it suffices to check that the right-hand side of (\ref{lemmalhs}) also satisfies the relation when substituted for the $\mu_y$. This amounts to
\begin{equation}\label{aboverelation1p5}
\frac{1}{(a;q)_y} \sum_{r=0}^{y} (-a)^r q^{mr} q^{\frac{r(r-1)}{2}} \frac{(q;q)_y}{(q;q)_r (q;q)_{y-r}} \sum_{i=0}^{y-r} C_{-a}(y-r,i) q^{im} = 1.
\end{equation}
In order to check that the above relation holds we can gather all coefficients associated with $q^{(r+i)m}$ and check that for $r+i=0$ the coefficients combine to equal 1, and for $r+i>0$ they combine to 0. Letting $n=r+i$, we must therefore check that
\begin{equation*}
\frac{(q;q)_y}{(a;q)_y} \sum_{r=0}^{y} (-a)^r \frac{q^{\frac{r(r-1)}{2}}}{(q;q)_r (q;q)_{y-r}}  C_{-a} (y-r,n-r)  = \begin{cases} 1& n=0,\\ 0 & n>0.\end{cases}
\end{equation*}
Using the definition of the $C$ coefficients the above relation reduces to
\begin{equation}\label{aboverelation2}
\frac{(q;q)_y}{(q;q)_{y-n}} \frac{(a;q)_{y-n}}{(a;q)_y} a^n \sum_{r=0}^{n} (-1)^r \frac{q^{\frac{r(r-1)}{2}}}{(q;q)_r (q;q)_{n-r}}= \begin{cases} 1& n=0,\\ 0 & n>0.\end{cases}
\end{equation}
From Corollary 10.2.2(c) of \cite{AAR} we find that
\begin{equation*}
\sum_{r=0}^{n} (-1)^r \frac{q^{\frac{r(r-1)}{2}}}{(q;q)_r (q;q)_{n-r}} = \begin{cases} 1& n=0,\\ 0 & n>0,\end{cases}
\end{equation*}
from which (\ref{aboverelation2}) immediately follows.
\end{proof}

\noindent {\bf Bernoulli case}: We will show the following stronger statement. For $j\geq 1$ and $t\geq 0$, let $\sigma^j_t$ the sigma-field generated by the random variables $x_{1}(t),\ldots, x_j(t)$. Then, as $\sigma^N_t$ measurable random variables
\begin{equation}\label{berncond}
\A{a_1\beta_{t+1}}{1} \cdots \A{a_N\beta_{t+1}}{N} \EE\left[\prod_{i=0}^{N} q^{(x_i(t+1)+i)y_i} \big\vert \sigma^N_t\right]
= \A{qa_1\beta_{t+1}}{1} \cdots \A{qa_N\beta_{t+1}}{N} \prod_{i=0}^{N} q^{(x_i(t+1)+i)y_i}.
\end{equation}
Taking expectations of both sides (with respect to the random variables $\vec{x}(t)$) recovers the desired result. In what follows we write $\beta$ instead of $\beta_{t+1}$.

The following lemma will quickly yield a proof of (\ref{berncond}).

\begin{lemma}\label{bernlem}
For $i=2,\ldots, N$,
\begin{eqnarray*}
\sum_{s_i=0}^{y_i} C_{a_i\beta}(y_i,s_i)\EE\left[ q^{(x_i(t+1)+i)(y_i-s_i)}q^{(x_{i-1}(t+1)+i-1)s_i} \big\vert \sigma^N_t, \sigma^{i-1}_{t+1}\right]\qquad\qquad\qquad\qquad&\\
= \sum_{s_i=0}^{y_i} C_{qa_i\beta} (y_i,s_i) q^{(x_i(t)+i)(y_i-s_i)}q^{(x_{i-1}(t)+i-1)s_i},\qquad\qquad\qquad\qquad\qquad&
\end{eqnarray*}
and for $i=1$
\begin{equation*}
C_{a_1\beta}(y_1,0)\EE\left[ q^{(x_1(t+1)+1)y_{1}} \big\vert \sigma^N_t\right] = C_{qa_1\beta} (y_1,0) q^{(x_1(t)+1)y_1}.
\end{equation*}
\end{lemma}
\begin{proof}
Let us first consider the $i=1$ case. Given the knowledge of $x_{1}(t)$, the dynamics of Bernoulli $q$-TASEP implies that with probability $\frac{a_1\beta}{1+a_1\beta}$, we have that $x_{1}(t+1)=x_{1}(t)+1$ and with probability $\frac{1}{1+\beta}$, we have that $x_{1}(t+1)=x_1(t)$. This implies that
\begin{eqnarray*}
C_{a_1\beta}(y_1,0)\EE\left[ q^{(x_1(t+1)+1)y_{1}} \big\vert \sigma^N_t\right] &=&  C_{a_1\beta}(y_1,0) q^{(x_1(t)+1)y_1}\left(\frac{a_1\beta}{1+a_1\beta} q^{y_1} + \frac{1}{1+a_1\beta}\right) \\
&=& C_{qa_1\beta} (y_1,0) q^{(x_1(t)+1)y_1},
\end{eqnarray*}
as desired. Here we used the third relation of (\ref{Cid}) with $s=0$ to reach the above conclusion.

Now consider the case when $i=2,\ldots, N$. Let $I_{i-1}$ represent the indicator function for the event that $x_{i-1}(t+1)=x_{i-1}(t)+1$ (i.e., particle $i-1$ jumped at time $t+1$). Note that this is measurable with respect to $\sigma^N_t,\sigma^{i-1}_{t+1}$. Thus, by virtue of the Bernoulli dynamics, we have that for $r,s\geq 0$
\begin{align*}
&\EE\left[ q^{(x_i(t+1)+i)r}q^{(x_{i-1}(t+1)+i-1)s} \big\vert \sigma^N_t, \sigma^{i-1}_{t+1}\right]\\
&=  I_{i-1} q^{(x_i(t)+i)r} q^{(x_{i-1}(t)+1+i-1)s} \left(\frac{a_i\beta}{1+a_i\beta} q^r + \frac{1}{1+a_i\beta}\right)\\
&\quad   + (1-I_{i-1}) q^{(x_i(t)+i)r} q^{(x_{i-1}(t)+i-1)s} \left((1-q^{x_{i-1}(t)-x_i(t)-1})\frac{a_i\beta}{1+a_i\beta} q^r + \frac{1+a_i\beta q^{x_{i-1}(t)-x_i(t)-1}}{1+a_i\beta}\right)\\
&= q^{(x_i(t)+i)r} q^{(x_{i-1}(t)+i-1)s} \frac{1}{1+a_i\beta} \left(I_{i-1} a_i\beta q^{r+s} + I_{i-1} q^s + (1-I_{i-1}) q^r + (1-I_{i-1})\right)\\
&\quad   + q^{(x_i(t)+i)(r-1)} q^{(x_{i-1}(t)+i-1)(s+1)} \frac{1}{1+a_i\beta} \left(a_i\beta(1-I_{i-1})(1-q^r)\right).
\end{align*}
Using the above, we may now compute
\begin{align*}
&\sum_{s_i=0}^{y_i} C_{a_i \beta}(y_i,s_i)\EE\left[ q^{(x_i(t+1)+i)(y_i-s_i)}q^{(x_{i-1}(t+1)+i-1)s_i} \big\vert \sigma^N_t, \sigma^{i-1}_{t+1}\right] \\
&= \sum_{s_i=0}^{y_i} q^{(x_i(t)+i)(y_i-s_i)}q^{(x_{i-1}(t)+i-1)s_i} \\
&\qquad\qquad\times \, \bigg(C_{a_i\beta}(y_i,s_i)\frac{1}{1+a_i\beta}\left(I_{i-1}a_i\beta q^{y_i} + I_{i-1} q^{s_i} + (1-I_{i-1})a_i\beta q^{y_i-s_i} + (1-I_{i-1})\right) \\
&\qquad\qquad\qquad\qquad\qquad\qquad\qquad\qquad\qquad+ C_{a_i\beta}(y_i,s_i-1) \frac{a_i\beta}{1+a_i\beta}(1-I_{i-1})(1-q^{y_i-s_i+1})\bigg).
\end{align*}
Observe that we may rewrite the above factor (in large parentheses) as
\begin{align*}
&\frac{1}{1+a_i\beta} \Big((1+a_i\beta q^{y_i-s_i}) C_{a_i\beta}(y_i,s_i) + a_i\beta(1-q^{y_i-s_i+1}) C_{a_i\beta}(y_i,s_i-1)\Big) \\
&\quad+ \frac{I_{i-1}}{1+a_i\beta} \Big(C_{a_i\beta}(y_i,s_i)(a_i\beta q^{y_i} + q^{s_i} - a_i\beta q^{y_i-s_i} -1) - C_{a_i\beta}(y_i,s_i-1)a_i\beta(1-q^{y_i-s_i+1})\Big)\\
&= C_{qa_i\beta}(y_i,s_i).
\end{align*}
This last equality came from two facts. The left-hand side has two terms. The first is equal to $C_{qa_i\beta}(y_i,s_i)$ by the third relation of (\ref{Cid}). The second term (with $I_{i-1}$) is zero by the first relation of (\ref{Cid}).

As a consequence of the above calculation we find that
\begin{align*}
&\sum_{s_i=0}^{y_i} C_{a_i\beta}(y_i,s_i)\EE\left[ q^{(x_i(t+1)+i)(y_i-s_i)}q^{(x_{i-1}(t+1)+i-1)s_i} \big\vert \sigma^N_t, \sigma^{i-1}_{t+1}\right]\\
& = \sum_{s_i=0}^{y_i}  C_{qa_i\beta}(y_i,s_i) q^{(x_i(t)+i)(y_i-s_i)}q^{(x_{i-1}(t)+i-1)s_i},
\end{align*}
as desired to complete the proof of the lemma.
\end{proof}

%The result of applying the operators on the right-hand side of the above equation shows that
%\begin{equation}\label{rhsabove}
%\textrm{RHS }(\ref{berncond}) =  C_{qa_1\beta}(y_1,0) \sum_{s_2=0}^{y_2} C_{qa_2\beta}(y_2,s_2) \cdots \sum_{s_N=0}^{y_N} C_{qa_N\beta}(y_N,s_N) \prod_{i=1}^{N} q^{(x_{i}(t)+i)(y_i -s_i + s_{i+1})}
%\end{equation}
%with the convention of $s_1=0$.

In order to conclude the proof of the Bernoulli case we use conditional expectations to rewrite the left-hand side of (\ref{berncond}) with $y_0=0$ as
\begin{align*}
&\A{a_1\beta}{1} \cdots \A{a_{N-1}\beta}{N-1} \A{a_N\beta}{N} \EE\Big[q^{(x_1(t+1)+1)y_1} \EE\Big[ q^{(x_2(t+1)+2)y_2} \cdots \\
&  \cdots \EE\Big[ q^{(x_{N-1}(t+1)+N-1)y_{N-1}} \EE\Big[ q^{(x_N(t+1)+N)y_N} \big\vert \sigma^N_t,\sigma^{N-1}_{t+1}\Big]  \big\vert \sigma^N_t,\sigma^{N-2}_{t+1}\Big]\cdots \big\vert \sigma^{N}_t, \sigma^{1}_{t+1} \Big]\big\vert \sigma^{N}_t\Big]\\
&= \A{a_1\beta}{1} \cdots \A{a_{N-1}\beta}{N-1} \EE\Big[q^{(x_1(t+1)+1)y_1} \EE\Big[ q^{(x_2(t+1)+2)y_2} \cdots \EE\Big[ q^{(x_{N-1}(t+1)+N-1)y_{N-1}} \\
& \quad\times\,\sum_{s_N=0}^{y_N} C_{a_N\beta}(y_N,s_N) \EE\Big[ q^{(x_N(t+1)+N)(y_N-s_N)}q^{(x_{N-1}(t+1)+N-1)s_N} \big\vert \sigma^N_t,\sigma^{N-1}_{t+1}\Big]\big\vert \sigma^N_t,\sigma^{N-2}_{t+1}\Big]\cdots \big\vert \sigma^{N}_t, \sigma^{1}_{t+1} \Big]\big\vert \sigma^{N}_t\Big]\\
&= \A{a_1\beta}{1} \cdots \A{a_{N-1}\beta}{N-1} \EE\Big[q^{(x_1(t+1)+1)y_1} \EE\Big[ q^{(x_2(t+1)+2)y_2} \cdots \EE\Big[ q^{(x_{N-1}(t+1)+N-1)y_{N-1}} \big\vert \sigma^N_t,\sigma^{N-2}_{t+1}\Big]\cdots \big\vert \sigma^{N}_t, \sigma^{1}_{t+1} \Big]\big\vert \sigma^{N}_t\Big]  \\
&\quad\times\,\sum_{s_N=0}^{y_N} C_{qa_N\beta} (y_N,s_N) q^{(x_N(t)+N)(y_N-s_N)}q^{(x_{N-1}(t)+N-1)s_N}
\end{align*}
where the first equality was from definition and the second equality followed from applying Lemma \ref{bernlem} with $i=N$. We may continue to use Lemma \ref{bernlem} to reduce the expression above, ultimately leading to
\begin{equation*}
\textrm{LHS }(\ref{berncond}) = C_{qa_1\beta}(y_1,0) \sum_{s_2=0}^{y_2} C_{qa_2\beta}(y_2,s_2)\cdots \sum_{s_N=0}^{y_N} C_{qa_N\beta}(y_N,s_N) \prod_{i=0}^{N} q^{(x_i(t)+i)(y_i+s_{i+1}-s_{i})}
\end{equation*}
with the convention $s_0=s_1=s_{N+1}=0$. On account of Remark \ref{zeroremark} we may now recognize that the right-hand side above, is equal to the right-hand side of (\ref{berncond}), as desired to complete the proof of the Bernoulli case.
\end{proof}

\section{Free evolution equations with $k-1$ boundary conditions}\label{freesec}
%
%
%
%\begin{example}\label{example1}
%We provide an example illustrating the meaning of this theorem and then give its proof. Consider the continuous time Poisson $q$-TASEP. Then the above theorem applied with $k=2$ shows that for $n_1>n_2$,
%\begin{eqnarray}
%\frac{d}{dt} \EE\left[q^{x_{n_1}(t)+n_1} q^{x_{n_2}(t)+n_2}\right] & = & (1-q)\left( \EE\left[q^{x_{n_1-1}(t)+n_1-1} q^{x_{n_2}(t)+n_2}\right]-\EE\left[q^{x_{n_1}(t)+n_1} q^{x_{n_2}(t)+n_2}\right]\right)\\
%& & + (1-q)\left( \EE\left[q^{x_{n_1}(t)+n_1} q^{x_{n_2-1}(t)+n_2-1}\right]-\EE\left[q^{x_{n_1}(t)+n_1} q^{x_{n_2}(t)+n_2}\right]\right);
%\end{eqnarray}
%whereas for $n_1=n_2$,
%\begin{equation}
%\frac{d}{dt} \EE\left[q^{x_{n_1}(t)+n_1} q^{x_{n_2}(t)+n_2}\right] = (1-q^2)\left( \EE\left[q^{x_{n_1}(t)+n_1} q^{x_{n_2-1}(t)+n_2-1}\right]-\EE\left[q^{x_{n_1}(t)+n_1} q^{x_{n_2}(t)+n_2}\right]\right).
%\end{equation}
%This system of ODEs is almost separable and constant coefficient -- if not for the behavior when $n_1=n_2$. In general, for $k\geq 2$ on encounters different right-hand sides for each possible clustering of coordinates of $\vec{n}$ together. Thus, it is a priori unclear how to go about finding explicit formulas for the solution to such systems.
%\end{example}
%
%

It is not a priori clear how to approach the problem of solving the true evolution equations of Theorem \ref{truethm} in closed form. Away from the boundary of $\Wk$ (i.e., when all of the elements of $\vec{n}$ are different, or equivalently all $y_i\leq 1$) the true evolution equation is constant coefficient and separable. This constant coefficient, separable equation can be extended to all of $(\Z_{\geq 0})^k$ and is what we will call the {\it free evolution equation}. The true and free evolution equations do not match near the boundary of $\Wk$ when there is clustering in $\vec{n}$. %A similar phenomena occurs for the discrete time geometric and Bernoulli true evolution equations (though the evolution equation away from the boundary is not separable, it is still suitably nice as we will see).

The following idea can be traced to Bethe's 1931 solution to the Heisenburg spin chain \cite{Bethe} and can be thought of as a generalization of the method of images or reflection principle. Bethe's idea is to try to rewrite the true evolution equation as the restriction of the free evolution equation subject to certain boundary conditions which, if satisfied, would imply that the restriction matched the true evolution equation. Generally the free evolution equation may be defined on a larger space (here $(\Z_{\geq 0})^k$) than is physically meaningful (here $\Wk$) and the initial data need only be imposed on the physically relevant portion of the space. If such a solution satisfying the free evolution equation, boundary conditions and initial data exists, then its restriction to the physically relevant space will necessarily solve the true evolution equation for the right initial data. This existence of solutions is not assured.

If the system in consideration is $k$-dimensional, then generically there could be boundary conditions corresponding to all possible compositions of $\{1,\ldots,k\}$. If in fact only the $k-1$ nearest neighbor boundary conditions are necessary, then the system is called {\it integrable} in the language of quantum many body systems. This is the case for the true evolution equations of Poisson, geometric and Bernoulli $q$-TASEP.

For the Poisson $q$-TASEP this reduction was observed in Proposition 2.7 of \cite{BCS}. One should note that it is not a priori clear that the boundary conditions for the discrete time $q$-TASEPs with time varying parameters $\alpha_t$ or $\beta_t$ would not depend on $t$. In fact, it turns out that this $t$ dependence of the true evolution equation plays no role in the form of the boundary condition. In fact, between all three versions of $q$-TASEP, the exact same boundary condition arise.

\begin{definition}
For a function $f:\Z\to \R$ define the operators $\dx{x}$ and $\nabla$ via
\begin{equation*}
(\dx{x}f)(n) = (1+x)f(n) - x f(n-1), \qquad (\nabla f)(n) = f(n-1) -f(n).
\end{equation*}
For a function $f:\Z^k\to \R$, let $[\dx{x}]_i$ and $[\nabla]_i$ act on the $i^{th}$ coordinate of $f$.
\end{definition}

The following theorem relates the true evolution equations for $q$-TASEP to free evolution equations with $k-1$ boundary conditions.

\begin{theorem}\label{freethm}
If $u(t;\vec{n})$ solves:
\begin{enumerate}
\item For all $\vec{n}\in (\Z_{\geq 0})^k$ and $t\geq 0$,
\begin{eqnarray*}
    \frac{d}{dt} u(t;\vec{n}) &=& \sum_{i=1}^{k} a_{n_i}(1-q)[\nabla]_i(t;\vec{n}) \qquad\qquad\qquad\quad\, \textrm{(Poisson)},\\
    u(t+1;\vec{n}) &=& [\dx{-a_{n_1} \alpha_{t+1} }]_1 \cdots [\dx{-a_{n_k} \alpha_{t+1}}]_k u(t;\vec{n}) \qquad\textrm{(geometric)},\\[.2em]
    \,[\dx{ a_{n_1}\beta_{t+1}} ]_1 \cdots [\dx{a_{n_k} \beta_{t+1}}]_k u(t+1;\vec{n}) &=& [\dx{q a_{n_1}\beta_{t+1}}]_1 \cdots [\dx{q a_{n_k}\beta_{t+1}}]_k u(t;\vec{n})\qquad\, \textrm{     (Bernoulli)};
\end{eqnarray*}
\item For all $\vec{n}\in (\Z_{\geq 0 })^k$ such that for some $i\in \{1,\ldots, k-1\}$, $n_{i}=n_{i+1}$,
    \begin{equation*}
    \big([\nabla]_i - q [\nabla]_{i+1}\big) u(t;\vec{n})=0;
    \end{equation*}
\item For all $\vec{n}\in (\Z_{\geq 0})^k$ such that $n_k=0$, $u(t;\vec{n}) \equiv 0$ for all $t\geq 0$;
\item For all $\vec{n}\in \Wk$, $u(0;\vec{n}) = h_0(\vec{y}(\vec{n}))$;
\end{enumerate}
then for all $\vec{y}\in Y_k^N$, $h(t;\vec{y}) = u(t;\vec{n}(\vec{y}))$ where $h$ is the solution to the true evolution equation for the (rate parameter $a_1,\ldots, a_N$) continuous time Poisson $q$-TASEP / discrete time Geometric $q$-TASEP (with time dependent jump parameters $\alpha_1,\alpha_2,\ldots\in (0,1)$) / discrete time Bernoulli $q$-TASEP (with time dependent jump parameters $\beta_1,\beta_2,\ldots \in (0,\infty)$), started from initial data $h_0(\vec{y})$.
\end{theorem}
\begin{proof}
It is immediate from the third and fourth hypotheses of Theorem \ref{freethm} that conditions (2) and (3) of Definition \ref{truedef} are satisfied. It suffices to check that condition (1) of the true evolution equation is satisfied by $u(t;\vec{n}(\vec{y}))$. We show this first for the Poisson case (which previously appeared in the proof of Proposition 2.7 of \cite{BCS}). We then deal simultaneously with the geometric and Bernoulli cases by using Lemma \ref{dblem} below. In what follows we assume that $\vec{n} = \vec{n}(\vec{y})$ and hence $\vec{y}= \vec{y}(\vec{n})$.

\noindent {\bf Poisson case:} Recall that the size of the cluster of elements of $\vec{n}$ equal to $i$ is given by $y_i$. Consider the cluster of elements equal to $N$: $n_1=n_2=\cdots =n_{y_{N}}$. Every other cluster works similarly to what we now describe and though it is possible that the cluster we study is empty, we may repeat the below calculation for any other cluster.

In order to prove the true evolution equation it suffices (by summing over all clusters) to show that
\begin{equation*}
(1-q) \sum_{i=1}^{y_N} a_N [\nabla]_i u(t;\vec{n}) = a_N (1-q^{y_N}) \nabla_{y_N} u(t;\vec{n}) = \LL{a_N}{N}u(t;\vec{y}),
\end{equation*}
where the second equality follows immediately from the definition of $\LL{a_N}{N}$. To see the first equality we use the boundary condition for the free evolution equation which implies that for $i\in \{1,\ldots, y_N\}$, $[\nabla]_i u(t;\vec{n}) = q^{y_N-i}[\nabla]_{y_N} u(t;\vec{n})$. Summing over $i$ gives the desired equality.

\noindent {\bf Geometric and Bernoulli cases:}
We use Lemma \ref{dblem} below to prove the equivalence of the free and true evolution equations for both the geometric and Bernoulli cases. However, we first provide a lemma which quickly leads to the proof of Lemma \ref{dblem}.

\begin{lemma}\label{prevlem}
Fix $y\geq 1$ and $a>0$. Assume that a function $f(n_1,\ldots, n_y)$ is such that if $n_i=n_{i+1}$ for any $i\in \{1,\ldots, y-1\}$ then $([\nabla]_i - q[\nabla]_{i+1})f(\vec{n}) = 0$. Then
\begin{equation}\label{desiredf}
\,[\dx{a}]_1 \cdots [\dx{a}]_y f(n,\ldots, n) = \sum_{s=0}^{y} C_{a}(y,s)  f(\underbrace{n,\ldots, n}_{y-s}, \underbrace{n-1,\ldots, n-1}_{s}).
\end{equation}
\end{lemma}
\begin{proof}
We prove this by induction in $y$. For $y=1$ observe that by definition of $[\dx{a}]$,
\begin{eqnarray*}
\,[\dx{a}]_1 f(n) &=& (1+a)f(n) - a f(n-1)\\
&=& C_a(1,0) f(n) + C_a(1,1) f(n-1),
\end{eqnarray*}
since $C_a(1,0)=(1+a)$ and $C_{a}(1,1) = -a$.

For $y>1$ assume that we have proved the lemma for $y-1$. Thus we have
\begin{eqnarray}\label{aboveplug}
\nonumber\,[\dx{a}]_1 [\dx{a}]_2 \cdots [\dx{a}]_y f(n,\ldots, n) &=& [\dx{a}]_1 \sum_{s=0}^{y-1} C_{a}(y-1,s)  f(n,\underbrace{n,\ldots, n}_{y-1-s}, \underbrace{n-1,\ldots, n-1}_{s})\\
&=& \sum_{s=0}^{y-1} C_{a}(y-1,s) (1+a) f(n,\underbrace{n,\ldots, n}_{y-1-s}, \underbrace{n-1,\ldots, n-1}_{s})\\
\nonumber&& + \sum_{s=0}^{y-1} C_{a}(y-1,s) (-a) f(n-1,\underbrace{n,\ldots, n}_{y-1-s}, \underbrace{n-1,\ldots, n-1}_{s}).
\end{eqnarray}

Using the relation $([\nabla]_i - q[\nabla]_{i+1})f(\vec{n}) = 0$ for $\vec{n}$ such that $n_i=n_{i+1}$, we see that
\begin{eqnarray*}
f(n-1,\underbrace{n,\ldots, n}_{y-1-s}, \underbrace{n-1,\ldots, n-1}_{s}) &=& q^{y-1-s} f(\underbrace{n,\ldots, n}_{y-1-s}, \underbrace{n-1,\ldots, n-1}_{s+1})\\
&& + (1-q^{y-1-s}) f(\underbrace{n,\ldots, n}_{y-s}, \underbrace{n-1,\ldots, n-1}_{s}).
\end{eqnarray*}
Plugging this into (\ref{aboveplug}) we arrive at
\begin{align*}
& \sum_{s=0}^{y-1} C_{a}(y-1,s) (1+a) f(n,\underbrace{n,\ldots, n}_{y-1-s}, \underbrace{n-1,\ldots, n-1}_{s})\\
& + \sum_{s=0}^{y-1} C_{a}(y-1,s) (-a) \Bigg(q^{y-1-s} f(\underbrace{n,\ldots, n}_{y-1-s}, \underbrace{n-1,\ldots, n-1}_{s+1})+ (1-q^{y-1-s}) f(\underbrace{n,\ldots, n}_{y-s}, \underbrace{n-1,\ldots, n-1}_{s})\Bigg).
\end{align*}
Grouping the coefficients of $f(\underbrace{n,\ldots, n}_{y-s}, \underbrace{n-1,\ldots, n-1}_{s})$ in the above summations and using the second relationship in (\ref{Cid}) we recover (\ref{desiredf}) and hence complete the inductive step.
\end{proof}

\begin{lemma}\label{dblem}
Fix $N,k\geq 1$ and $a_1,\ldots, a_N>0$. Assume that a function $f(n_1,\ldots, n_y)$ is such that if $n_i=n_{i+1}$ for any $i\in \{1,\ldots, y-1\}$ then $([\nabla]_i - q[\nabla]_{i+1})f(\vec{n}) = 0$. Then for all $n_1\geq n_2\geq \cdots \geq n_k>0$,
\begin{equation*}
\,[\dx{a_{n_1}}]_1 \cdots [\dx{a_{n_k}}]_k  f(n_1,\ldots, n_k) = \sum_{s_1=0}^{y_1} C_{a_1}(y_1,s_1) \cdots  \sum_{s_N=0}^{y_N} C_{a_N}(y_N,s_N) f(\vec{n}(\vec{y}^{\,\vec{s}})).
\end{equation*}
%where $\vec{y} = \vec{y}(\vec{n})$ and $\vec{y}^s = (y_0+s_1,y_1+s_2-s_1,y_2+s_3-s_2,\ldots, y_{N-1}+s_N-s_{N-1},y_{N}-s_{N})$.
\end{lemma}
\begin{proof}
Observe that $\vec{n}$ can be split into clusters of equal variables so that $n_1= \ldots =n_{y_N}=N$ and $n_{y_N+1}=\cdots = n_{y_N+y_{N-1}}=N-1$ and so on (recall that $\vec{y}=\vec{y}(\vec{n})$ gives the sizes of these clusters). Due to the ordering of $\vec{n}$, the result we seek to prove follows readily from applying Lemma \ref{prevlem} to each cluster of $\vec{n}$ starting with the cluster of $n_i$ which equal 1, and ending with the cluster of $n_i$ equal to $N$.
\end{proof}
Let us now complete the proof of Theorem \ref{freethm}.

We may apply Lemma \ref{dblem} (given below) to the solution $u(t;\vec{n})$ of the geometric / Bernoulli free evolution equation (the boundary condition on $u$ implies the hypotheses of the lemma are met). Since $u(t;\vec{n}) \equiv 0$ for $n_k=0$, it follows that the summation over $s_1$ in the outcome of the application of the lemma should be removed and only the $s_1=0$ term remain (that is, because all terms involving $s_1>0$ are necessarily zero). Comparing the outcome to the (respective) true evolution equation of Definition  \ref{truedef} one finds that they match.
\end{proof}

The following is an immediate corollary of the combination of Theorems \ref{freethm} and \ref{truethm}.
\begin{corollary}\label{freethmcor}
Consider the (rate parameter $a_1,\ldots, a_N$) continuous time Poisson $q$-TASEP / discrete time Geometric $q$-TASEP (with time dependent jump parameters $\alpha_1,\alpha_2,\ldots\in (0,1)$) / discrete time Bernoulli $q$-TASEP (with time dependent jump parameters $\beta_1,\beta_2,\ldots\in (0,\infty)$) started from an arbitrary initial condition $\vec{x}(0) = \vec{x}$. Then for any $k\geq 1$ and $n_1\geq n_2\geq \cdots \geq n_k>0$,
\begin{equation*}
\EE\left[\prod_{i=1}^k q^{x_{n_i}(t)+n_i}\right]  = u(t;\vec{n})
\end{equation*}
where $u(t;\vec{n})$ solves the (respective) free evolution equation with $k-1$ boundary conditions started from initial data
\begin{equation*}
u(0;\vec{n}) = \EE\left[\prod_{i=1}^k q^{x_{n_i}+n_i}\right].
\end{equation*}
\end{corollary}

%\begin{remark}
%Remark on how one might try to write a discrete time bose gas.
%\end{remark}

\section{Checking the nested contour integral formulas}\label{checkingformulas}

We now conclude the proof of Theorem \ref{momthm} by showing that the right-hand side of (\ref{momthmeq}) solves the free evolution equations with $k-1$ boundary conditions which is given in Theorem \ref{freethm}. Corollary \ref{freethmcor} then immediately implies Theorem \ref{momthm}. For ease of readability let us recall the right-hand side of (\ref{momthmeq}) and denote it as
\begin{equation}\label{momthmeq2}
m(t;\vec{n}) = \frac{(-1)^k q^{\frac{k(k-1)}{2}}}{(2\pi \iota)^k} \int \cdots \int \prod_{1\leq A<B\leq k} \frac{z_A-z_B}{z_A-qz_B} \prod_{j=1}^{k} \prod_{i=1}^{n_j}\frac{a_i}{a_i-z_j} \frac{f^{\ell}(qz_j)}{f^{\ell}(z_j)} \frac{dz_j}{z_j}.
\end{equation}

There are four hypotheses to check in Theorem \ref{freethm}. The first hypothesis to check is that $m(t;\vec{n})$ solves the free evolution equation corresponding with the choice of $\ell\in \big\{{\rm Poi}, {\rm geo}, {\rm Ber}\big\}$. Since the free evolution equation is linear and separable, it suffices to check that for arbitrary $z$ and $n$,
$$
\prod_{i=1}^{n} \frac{a_i}{a_i-z} \frac{f^{\ell}(qz)}{f^{\ell}(z)}
$$
satisfies the $k=1$ version of the respective free evolution equations.

\noindent {\bf Poisson case:} We must check that
$$
\frac{d}{dt} \prod_{i=1}^{n} \frac{a_i}{a_i-z} e^{(q-1)zt}  = a_n(1-q) \nabla \prod_{i=1}^{n} \frac{a_i}{a_i-z} e^{(q-1)zt}.
$$
The $d/dt$ brings out a factor $(q-1)z$ on the right-hand side, while the $\nabla$ brings out a factor $\tfrac{a_n-z}{a_n} - 1$. Comparing both sides we find equality.

\noindent {\bf Geometric case:} We must check that
$$
\prod_{i=1}^{n} \frac{a_i}{a_i-z} \prod_{s=1}^{t+1} (1-\alpha_s z)  = \dx{-a_n\alpha_{t+1}} \prod_{i=1}^{n} \frac{a_i}{a_i-z} \prod_{s=1}^{t} (1-\alpha_s z).
$$
The $\dx{-a_n\alpha_{t+1}}$ brings out a factor $\big((1-a_n \alpha_{t+1}) + a_n\alpha_{t+1} \tfrac{a_n-z}{a_n}\big)$ which simplifies to $(1-\alpha_{t+1} z)$, and hence we find equality.

\noindent {\bf Bernoulli case:} We must check that
$$
\dx{a_n\beta{t+1}} \prod_{i=1}^{n} \frac{a_i}{a_i-z} \prod_{s=1}^{t+1} \frac{1+q \beta_s z}{1+\beta_s z}  = \dx{q a_n\beta{t+1}} \prod_{i=1}^{n} \frac{a_i}{a_i-z} \prod_{s=1}^{t} \frac{1+q \beta_s z}{1+\beta_s z}.
$$
The $\dx{a_n\beta{t+1}}$ on the left-hand side brings out a factor $\big((1+a_n \beta_{t+1}) -a_n\beta_{t+1} \frac{a_n-z}{a_n}\big)$ which simplifies to $(1+\beta_{t+1} z)$, while the $\dx{q a_n\beta{t+1}} $ on the right-hand side brings out a factor $\big((1+qa_n \beta_{t+1}) -qa_n\beta_{t+1} \frac{a_n-z}{a_n}\big)$ which simplifies to $(1+q\beta_{t+1} z)$. Comparing both sides we find equality.

The second hypothesis to check is that the boundary condition is satisfied. Without loss of generality assume $n_1=n_2$ (general $n_i=n_{i+1}$ is identical). We wish to show that
$$\left([\nabla]_1-q[\nabla]_2\right) m(t;\vec{n}) = 0.$$
Applying the operator $[\nabla]_1-q[\nabla]_2$ to the integrand of $m(t;\vec{n})$ brings out a factor of $-(z_1-qz_2)$. This new factor cancels the denominator $(z_1-qz_2)$. On account of this we may freely (without encountering any poles) deform the contours for $z_1$ and $z_2$ to coincide. Hence we may write
$$\left([\nabla]_1-q[\nabla]_2\right) m(t;\vec{n}) = \int \int (z_1-z_2) G(z_1)G(z_2)dz_1 dz_2$$
where the function $G(z)$ involves the integrals in $z_3,\ldots, z_k$. Since the two contours are identical, this integral is clearly zero, as desired.

The third hypothesis to check is that for $n_k=0$, $m(t;\vec{n}) = 0$. This follows from simple residue calculus since when $n_k=0$ there is no pole at $z_k=1$ and since this was the only pole contained by the $z_k$ contour, by Cauchy's theorem the $z_k$ integral equals 0.

The final hypothesis to check is the initial data. Since we are dealing with step initial condition $x_i(0)=-i$, it follows that $q^{x_i(0)+i} \equiv 1$ and hence we must show that $m(0;\vec{x}) \equiv 1$. This follows from residue calculus as well. Expand the $z_1$ contour to infinity. Since $t=0$, $f^{\ell}_0(z)=1$ and so the only pole in $z_1$ is encountered at $z_1=0$ ($z_1=\infty$ is not a pole because of the decay coming from $1/(1-z_1)$). Because we pass the pole at 0 from the outside, the contribution of the residue is $-q^{-(k-1)}$ times the same integral, but with $k-1$ variables. Repeating this procedure yields the desired result.

This completes the proof of Theorem \ref{momthm}.

\section{Relation to duality and Macdonald processes}\label{endsec}

\subsection{True evolution equations and duality}\label{dualityrem}
At least for the Poisson and geometric $q$-TASEP, the true evolution equations are closely related to Markov process dualities. This was understood for Poisson $q$-TASEP in \cite{BorCor}. Two Markov processes $x(t)$ and $y(t)$  (with state spaces $X$ and $Y$ respectively) are said to be dual with respect to a function $H:X\times Y \to \R$ if for all $x\in X$, $y\in Y$ and $t\geq 0$,
\begin{equation*}
\EE^x\left[H(x(t);y)\right] = \EE^{y} \left[H(x;y(t))\right]
\end{equation*}
where $\EE^x$ and $\EE^y$ refer to the expectations of the respective Markov chains $x(t)$ and $y(t)$ started from $x(0)=x$ and $y(0)=y$.

Let us recall the duality for the Poisson $q$-TASEP. Let $\vec{y}(t)$ be the totally asymmetric zero range process with state space $Y^N$ from (\ref{yn}) in which the rate at which a particle moves from site $i$ to $i-1$ (for $i=1,\ldots, N$) is given by $a_i(1-q^{y_i})$. Notice that this is the same jump rate as in Remark \ref{zerorem} though presently there are no sources or sinks. Let $\vec{x}(t)$ be the continuous time Poisson $q$-TASEP with $N$ particles. Then it is easily checked that $\vec{x}(t)$ and $\vec{y}(t)$ are dual with respect to
\begin{equation*}
H(\vec{x};\vec{y}) = \prod_{i=0}^N q^{(x_{i}+i)y_i}
\end{equation*}
where for $y_0=0$ the product is over $i\geq 1$ and for $y_0>0$, the product is taken to be zero (due to the virtual particle $x_0=+\infty$). The true evolution equation follows immediately from this duality since
\begin{equation*}
\frac{d}{dt} \EE^x\left[H(\vec{x}(t);\vec{y})\right] = \frac{d}{dt} \EE^y\left[H(\vec{x};\vec{y}(t))\right] = \sum_{i=1}^{N} \LL{a}{i} \left[H(\vec{x};\vec{y}(t))\right].
\end{equation*}
Here, the first equality is from duality and the second one is from the Kolmogorov forward equation and the fact that  $\sum_{i=1}^{N} \LL{a}{i}$ is the generator of the dual zero range process.

It is also possible to relate the discrete time geometric $q$-TASEP to a discrete time totally asymmetric zero range process $\vec{y}(t)$ with state space $Y^N$. In parallel, the state of $\vec{y}(t)$ is updated to that of $\vec{y}(t+1)$ via the following procedure which occurs at each site $i=1,\ldots, N$. At time $t$ there are $y_i(t)$ particles lined up above site $i$. Start with the bottom-most particle and with probability $a_i\alpha_{t+1}$ move the particle to site $i-1$ at time $t+1$. Then proceed sequentially by considering the second particle from bottom, and then the third and so on. For the $j^{th}$ particle from bottom, with probability $a_i\alpha_{t+1}q^{m_j}$ move it to site $i-1$ at time $t+1$. Here $m_j$ counts the number of particles in site $i$ which are below the $j^{th}$ particle and which have been tagged to move to site $i-1$ at time $t+1$. The update rule is sequential within each site $i$, but parallel amongst different sites.

This duality can be proved using similar considerations as in the above proof of the geometric true evolution equation. Since it is not necessary for the purposes of this paper, we leave this just as a remark and do not present this proof.

It is not presently clear how to relate the true evolution equation for the discrete time Bernoulli $q$-TASEP to a duality for that system.

\subsection{True evolution equations and Macdonald difference operators}\label{MacSec}
(Ascending) Macdonald processes (see Section 2 and 3 of \cite{BorCor} for more details) are measures on triangular arrays of interlacing nonnegative integers
$$\lambda = \big\{\lambda^{(j)}_{i}, 1\leq i\leq j\leq N : \lambda^{(j)}_{i} \leq \lambda^{(j-1)}_{i-1}\leq \lambda^{(j)}_{i-1}\big\}.$$
We write $\lambda^{(j)}$ for the $j^{th}$ level of the triangle, and note that $\lambda^{(j)}$ is a partition of length at most $j$. The measure on $\lambda$ is specified by $a_1,\ldots, a_N>0$, $\{\alpha_1,\alpha_2,\ldots\}$, $\{\beta_1,\beta_2,\ldots,\}$, $\gamma\geq 0$ and two parameters $q,t\in (0,1)$ as follows:
\begin{equation*}
\PP(\lambda) = \frac{P_{\lambda^{(1)}}(a_1) P_{\lambda^{(2)}/\lambda^{(1)}}(a_2) \cdots P_{\lambda^{(N)}/\lambda^{(N-1)}}(a_N) Q_{\lambda^{(N)}}(\rho)}{\Pi(a_1,\ldots, a_N;\rho)},
\end{equation*}
and the marginal of this measure on a single level $\lambda^{(N)}$ is given by
\begin{equation*}
\PP(\lambda^{(N)}) = \frac{P_{\lambda^{(N)}}(a_1,\ldots, a_N) Q_{\lambda^{(N)}}(\rho)}{\Pi(a_1,\ldots, a_N;\rho)}.
\end{equation*}
In the above, $P$ and $Q$ are Macdonald symmetric functions (indexed by partitions or skew partitions) and $\rho$ is a Macdonald positive specialization (i.e., a homomorphism from the algebra of symmetric functions to $\R$ that sends Macdonald symmetric functions to elements of $\R_{\geq 0}$) specified via the generating function
\begin{equation*}
\Pi(u;\rho) = \sum_{j=0}^{\infty} u^j Q_{(j)}(\rho) = e^{\gamma u} \prod_{i\geq 0}\frac{(t\alpha_i u;q)_{\infty}}{(\alpha_i u;q)_{\infty}} (1+\beta_i u).
\end{equation*}
For such a specialization and positive $a_i$, the numerator in the definition of $\PP$ above is non-negative for all triangular arrays $\lambda$, and as long as the normalizing constant $\Pi(a_1,\ldots, a_N;\rho) := \Pi(a_1,\rho) \cdots \Pi(a_N;\rho)$ is finite, the above expression describes a probability measure.

For a partition $\mu$ of length $N$, the Macdonald polynomial $P_{\mu}$ is an eigenfunction for the first Macdonald difference operator
\begin{equation*}
D_N = \sum_{i=1}^{N} \prod_{j\neq i} \frac{tx_i-x_j}{x_i-x_j}  T_{q,x_i}
\end{equation*}
(where $T_{q,x_i}f(x_1,\ldots ,x_N) = f(x_1,\ldots, qx_i,\ldots, x_N)$) with eigenvalue
\begin{equation}
q^{\mu_1} t^{N-1} +q^{\mu_2}t^{N-2} + \cdots q^{\mu_N} t^0.
\end{equation}
Notice that when $t=0$ the eigenvalue is simply $q^{\mu_N}$. We will assume henceforth that $t=0$ (though everything except the commutation relations holds for general $t$).

The normalizing constant is given by
$$\Pi(a_1,\ldots, a_N;\rho) = \sum_{\lambda^{(N)}} P_{\lambda^{(N)}}(a_1,\ldots,a_N)Q_{\lambda^{(N)}}(\rho).$$
It was observed in \cite{BorCor} that by linearity of $D_N$ and by the above eigenvalue relation, for any $y\geq 0$
$$
\frac{(D_N)^y \Pi(x_1,\ldots,x_N;\rho)}{\Pi(x_1,\ldots,x_N;\rho)}\Big\vert_{x_1=a_1,\ldots, x_N=a_N}  = \EE\big[ q^{\lambda^{(N)}_{N} y} \big].
$$
A generalization of this was provided in \cite{BCGS}, showing that if $D_{j}$ represents the first difference operator acting on $x_1,\ldots, x_j$, then for any $y_1,\ldots, y_N\geq 0$
$$
\frac{(D_1)^{y_1} \cdots (D_N)^{y_N} \Pi(x_1,\ldots,x_N;\rho)}{\Pi(x_1,\ldots,x_N;\rho)}\Big\vert_{x_1=a_1,\ldots, x_N=a_N}  = \EE\left[\prod_{i=1}^{N} q^{\lambda^{(i)}_i y_i} \right].
$$

Since the $D_j$ and $\Pi$ are explicit and suitably nice, it is possible to encode the above repeated application of difference operators onto $\Pi$ in terms of a nested contour integral formula (one recovers the difference operators by expansion of the formula into residues). It is shown in \cite{BCGS} that for $\vec{y}=(0,y_1,\ldots,y_N)$ with $\sum_{j=1}^{N} y_j = k$, and $\vec{n}= \vec{y}(\vec{n})$,
\begin{equation*}
\EE\left[\prod_{i=1}^{N} q^{\lambda^{(i)}_i y_i} \right]  = \frac{(-1)^k q^{\frac{k(k-1)}{2}}}{(2\pi \iota)^k} \int \cdots \int \prod_{1\leq A<B\leq k} \frac{z_A-z_B}{z_A-qz_B} \prod_{j=1}^{k} \left(\prod_{i=1}^{n_j} \frac{a_i}{a_i-z_j}\right) \frac{\Pi(qz_j;\rho)}{\Pi(z_j;\rho)} \frac{dz_j}{z_j}.
\end{equation*}
where the contour of integration for $z_A$ contains $a_1,\ldots, a_N$, and $\{qz_B\}_{B>A}$ but not 0 or any other poles of the integrand.

Plugging in the formula for $\Pi$ (given above in terms of $\alpha$'s, $\beta$'s and $\gamma$) it is immediately clear that
\begin{equation}\label{lambdax}
\EE\left[\prod_{i=1}^{N} q^{\lambda^{(i)}_i y_i} \right] = \EE\left[\prod_{i=1}^{N} q^{(x_i(t)+i)y_i} \right],
\end{equation}
where the expression on the right-hand side corresponds to the version of $q$-TASEP considered in Remark \ref{arborder}. Therefore, since these moments characterize the joint law of $\{\lambda^{(i)}_i\}_{i=1}^N$ and $\{x_{i}(t)+ i\}_{i=1}^{N}$, it follows that these two sets are equal in distribution.

There are two immediate questions this equality prompts. First is whether there is a way to embed the various $q$-TASEP dynamics into a Markov chain on interlacing triangular arrays so that this equality becomes natural. Second is how the true evolution equation (which the right-hand side of the above equality satisfies) can be seen directly from the left-hand side as a consequence of Macdonald difference operator relations (without any reference to $q$-TASEPs). We answer both of these questions for continuous time Poisson $q$-TASEP first, and then indicate some extensions of the answer to the second question for the two discrete time versions as well.

There exist Markov dynamics on the space of interlacing triangular arrays which preserve the class of Macdonald processes. One example of such a dynamic was introduced and studied in Sections 2 and 3 of \cite{BorCor} (see \cite{OConPei,BorPet} for other such dynamics). When the parameter $t=0$ the continuous time dynamics on the triangular arrays becomes quite simple (cf. Section 3.3 of \cite{BorCor}). Each $\lambda^{(j)}_{i}$ attempts to increase its value by one according to independent exponentially distributed jumping times with rate given by
$$
a_j \frac{(1-q^{\lambda^{(j-1)}_{i-1} - \lambda^{(j)}_{i}}) (1-q^{\lambda^{(j)}_{i} - \lambda^{(j)}_{i+1}+1})}{(1-q^{\lambda^{(j)}_{i} - \lambda^{(j-1)}_{i}})},
$$
where terms involving indices outside of the triangular array are omitted.
The simplest Macdonald process with which to initiate $\lambda$ is specified by taking all $\alpha_i=\beta_i=\gamma=0$ and corresponds with taking $\lambda^{(j)}_{i}\equiv 0$. Then if the above dynamics are run for time $t$ (not to be confused with the Macdonald parameter $t$ which has been fixed to be zero) then $\lambda(t)$ is distributed according to a Macdonald process specified by taking all $\alpha_i=\beta_i=0$ but $\gamma=t$.

The above dynamics have the property that the evolution of $\{\lambda^{(i)}_i\}_{i=1}^{N}$ is marginally Markovian (with respect to its own filtration). In fact, calling $\lambda^{(i)}_i = x_{i}(t)+ i$ we readily observe that $x_i(t)$ evolves according to continuous time Poisson $q$-TASEP with particle jump parameters $a_1,\ldots, a_N$. It was through this line of reasoning that Poisson $q$-TASEP was first introduced and through the above Macdonald difference operator considerations that the formulas for moments of $q^{x_{i}(t)+i}$ were initially calculated.

In light of equation (\ref{lambdax}) we are led to the second question, of how the true evolution equation for Poisson $q$-TASEP can be seen directly from the language of Macdonald difference operators. In particular, we would like to show that
$$h(t;\vec{y})=\EE\left[\prod_{i=1}^{N} q^{(x_i(t)+i) y_i} \right]$$
solves
\begin{equation}\label{lambdatrue}
\frac{d}{dt} h(t;\vec{y})  = \sum_{i=1}^{N} a_i (1-q^{y_i}) \big(h(t;\vec{y}^{\,1}_{i,i-1})-h(t;\vec{y})\big).
\end{equation}

We claim that the fact that $h(t;\vec{y})$ solves the true evolution equation can be seen as a consequence of a commutation relation for Macdonald first difference operators, as well as the fact that by the above discussion
$$h(t;\vec{y})= \frac{(D_1)^{y_1} \cdots (D_N)^{y_N} e^{(x_1+\cdots +x_N)t}}{e^{(x_1+\cdots +x_N)t}}\Big\vert_{x_1=a_1,\ldots, x_N=a_N}.$$

The commutation relation is given as part 1 of the following lemma. For parts 2 and 3 it is convenient to define
$$
\mathcal{\tilde{A}}^{(a,k)}_n =\sum_{s=0}^{k} C_{a}(k,s)(D_{n-1})^s (D_{n})^{k-s},
$$
where $C_{a}(k,s)$ were defined in (\ref{Cas}).

\begin{lemma}\label{commlemma}
\mbox{}
\begin{enumerate}
\item For any $k\geq 0$ and any $1\leq n\leq N$,
\begin{equation}\label{commute}
\left[(D_n)^k,\sum_{i=1}^{n} x_i \right] = (1-q^k) x_n (D_{n-1}-D_n) (D_n)^{k-1}.
\end{equation}
\item For any $k\geq 0$, $1\leq n\leq N$ and $\alpha\in (0,1)$,
\begin{equation*}
(D_n)^k \prod_{i=1}^{n}\frac{1}{(\alpha x_i;q)_{\infty}} = \prod_{i=1}^{n}\frac{1}{(\alpha x_i;q)_{\infty}} \mathcal{\tilde{A}}^{(-\alpha x_n,k)}_n.
\end{equation*}
\item
For any $k\geq 0$, $1\leq n\leq N$ and $\beta\in (0,\infty)$,
\begin{equation*}
 \mathcal{\tilde{A}}^{(\beta x_n,k)}_n \prod_{i=1}^{n}(1+\beta x_i)=\prod_{i=1}^{n}(1+\beta x_i)  \mathcal{\tilde{A}}^{(q\beta x_n,k)}_n .
\end{equation*}
\end{enumerate}
\end{lemma}
\begin{proof}
We first remark that in the statement of the lemma, $\sum_{i=1}^{n} x_i$, $\prod_{i=1}^{n}\frac{1}{(\alpha x_i;q)_{\infty}}$ and $\prod_{i=1}^{n}(1+\beta x_i)$ should all be treated as operators which act by multiplication. Thus, we seek to prove the above identities as operators applied to general functions $F(x_1,\ldots, x_n)$. By polarization it suffices to prove these identities when applied to functions of the form $F(x_1,\ldots, x_n) = f(x_1)\cdots f(x_n)$. We prove part 1 of the lemma and simply note that analogous considerations (or comparison to the treatment of the geometric and Bernoulli many-body systems earlier in this paper) suffice to prove parts 2 and 3.

We can rewrite the desired identity we wish to prove as (adding harmless $1/F(x_1,\ldots,x_n)$ factors to both sides)
\begin{align*}
&\frac{d}{dt}\, \frac{1}{F(x_1,\ldots,x_n)} e^{-t\sum_{i=1}^{n} x_i} (D_n)^k e^{t\sum_{i=1}^{n} x_i} F(x_1,\ldots,x_n) \Big\vert_{t=0}\\
&=  \frac{1}{F(x_1,\ldots,x_n)}(1-q^k) x_n (D_{n-1}-D_n) (D_n)^{k-1} F(x_1,\ldots,x_n)
\end{align*}
for $F(x_1,\ldots, x_n) = f(x_1)\cdots f(x_n)$. This may be expressed in terms of nested contour integrals. The left-hand side becomes
\begin{equation}\label{exp1s}
\frac{(-1)^k q^{\frac{k(k-1)}{2}}}{(2\pi \iota)^k} \int \cdots \int \prod_{1\leq A<B\leq k} \frac{z_A-z_B}{z_A-qz_B} \left(\sum_{i=1}^{k} z_i\right) \prod_{j=1}^{k} \prod_{i=1}^{n} \frac{x_i}{x_i -z_j} \frac{f(qz_j)}{f(z_j)} \frac{dz_j}{z_j}
\end{equation}
while the right-hand side becomes
\begin{equation}\label{exp2s}
\frac{(-1)^k q^{\frac{k(k-1)}{2}}}{(2\pi \iota)^k}\int \cdots \int \prod_{1\leq A<B\leq k} \frac{z_A-z_B}{z_A-qz_B} \left((1-q^k) x_n \left(\frac{x_n-z_k}{x_n} -1\right)\right) \prod_{j=1}^{k} \prod_{i=1}^{n} \frac{x_i}{x_i -z_j} \frac{f(qz_j)}{f(z_j)} \frac{dz_j}{z_j}.
\end{equation}
In both sets of integrals we assume that the $z_A$ contour can be chosen so as to contain $qz_B$ for $B>A$ as well as $\{x_1,\ldots, x_n\}$, but not 0 or any poles of $f(qz)/f(z)$. Observe that due to the choice of nested contours, it follows that for $1\leq i\leq y-1$ and any symmetric function $G(z_1,\ldots, z_k)$ (which does not introduce new poles into the below integrand)
\begin{equation*}
\int\cdots \int \prod_{1\leq A<B\leq k} \frac{z_A-z_B}{z_A-qz_B} (z_i - q z_{i+1}) G(z_1,\ldots, z_k)=0.
\end{equation*}
Using this, we find that the term $\sum_{i=1}^{k} z_i$ in the integrand of (\ref{exp1s}) can be turned into $(1-q^k)(-z_k)$ which exactly matches the term in the integrand of (\ref{exp2s}) and thus proves the lemma.
\end{proof}

To conclude, let us explain how the commutation relations imply the true evolution equation. Observe that
\begin{equation*}
\frac{d}{dt} h(t;\vec{y}) = -p_1 e^{-tp_1} (D_1)^{y_1} \cdots (D_N)^{y_N} e^{t p_1} + e^{-tp_1} (D_1)^{y_1}\cdots (D_{N})^{y_N} p_1 e^{tp_1}
\end{equation*}
where $p_1=\sum_{i=1}^{N} x_i$, and on the right-hand side we assign $x_1=a_1,\ldots,x_N=a_N$ after applying all of the operators (we suppress writing this assignment above and in what follows). The commutation relation allows us to pull the $p_1$ in the second term from the right side of the expression to the left side. In the first application of (\ref{commute}) we replace
$$
(D_N)^{y_N} p_1 = p_1 (D_N)^{y_N} + x_N (1-q^{y_N}) (D_{N-1} - D_N) (D_N)^{y_N-1}
$$
and find
\begin{eqnarray*}
\frac{d}{dt} h(t;\vec{y}) &=& -p_1 e^{-tp_1} (D_1)^{y_1} \cdots (D_N)^{y_N} e^{t p_1} + e^{-tp_1} (D_1)^{y_1}\cdots (D_{N-1})^{y_{N-1}} p_1 (D_{N})^{y_N} e^{tp_1}\\ &&+ x_N(1-q^{y_N}) e^{-tp_1} (D_1)^{y_1} \cdots (D_{N-1})^{y_{N-1}} (D_{N-1}-D_N) (D_N)^{y_N-1}e^{tp_1}.
\end{eqnarray*}
Repeating this $N$ times we find
\begin{equation*}
\frac{d}{dt} h(t;\vec{y}) = \sum_{i=1}^{N} x_i (1-q^{y_i}) e^{-tp_1} (D_1)^{y_1} \cdots (D_{i-1}-D_i) (D_{i})^{y_i} \cdots (D_{N})^{y_N}e^{tp_1},
\end{equation*}
the right-hand side of which is readily matched to that of the true evolution equation.

\end{document}